\documentclass{amsart}
\usepackage{enumerate}
\usepackage{amssymb}
\usepackage{mathtools}
\usepackage{mathrsfs}
\usepackage{color}
\usepackage[
colorlinks=true,
linkcolor=blue,
citecolor=blue,
urlcolor=blue]{hyperref}

\usepackage{changes}

\usepackage{bbm, dsfont}
\numberwithin{equation}{section}

\newtheorem{theorem}{Theorem}[section]
\newtheorem{conjecture}[theorem]{Conjecture}
\newtheorem{lemma}[theorem]{Lemma}
\newtheorem{definition}[theorem]{Definition}
\newtheorem{proposition}[theorem]{Proposition}
\newtheorem{remark}[theorem]{Remark}

\newtheorem{corollary}[theorem]{Corollary}

\newcommand{\E}{{\mathbb E}}

\newcommand{\M}{{\mathbb M}}
\newcommand{\Z}{{\mathbb Z}}

\newcommand{\dc}{\mathbf{d}}
\newcommand{\La}{\mathbf{\mathcal{L}}}
\newcommand{\tr}{\mathrm{tr}}
\newcommand{\MP}{\mathbb{M}_{2^{n}}}

\newcommand{\s}{\mathbf{s}}
\newcommand{\Ex}{\mathcal{E}}

\newcommand{\wyd}{W_{\thickapprox d}}

\newcommand{\zfourn}{\{0,1,2,3\}^n}
\newcommand{\supp}{{\mathrm{supp}}}
\newcommand{\var}{{\mathrm{var}}}

\title{Quantum KKL-type inequalities revisited}

\author[Jiao]{Yong Jiao}
\address{School of Mathematics and Statistics, Central South University, HNP-LAMA, Changsha 410075, China}
\email{jiaoyong@csu.edu.cn}

\author[Lin]{Wenlong Lin}
\address{School of Mathematics and Statistics, Central South University, HNP-LAMA, Changsha 410075, China}
\email{linwenlong2000@foxmail.com}

\author[Luo]{Sijie Luo}
\address{School of Mathematics and Statistics, Central South University, HNP-LAMA, Changsha 410075, China}
\email{sijieluo@csu.edu.cn}

\author[Zhou]{Dejian Zhou}
\address{School of Mathematics and Statistics, HNP-LAMA, Central South University, Changsha 410075, China}
\email{zhoudejian@csu.edu.cn}


\begin{document}
\maketitle

\begin{abstract}
In the present paper, we develop the random restriction method in the quantum framework. By applying this method, we establish the quantum Eldan-Gross inequality, the quantum Talagrand isoperimetric inequality, and related quantum KKL-type inequalities. Our results recover some recent results of Rouz\'e et al. \cite{RWZ2024} and Jiao et al. \cite{JLZ2025}, which can be viewed as alternative answers to the quantum KKL conjecture  proposed by Motanaro and Osborne in \cite{MO2010}.
\end{abstract}

\section{Introduction}
Motivated by problems from complexity theory, geometric functional analysis, and computer science, numerous remarkable results have been developed in the hypercube framework, making Boolean analysis one of the most active areas in discrete Fourier analysis, combinatorial optimization, and related fields in the past decades. To state results, we begin with recalling basic concepts and notions in hypercube setting. For fixed $n\in\mathbb{N}$, let $\{-1,1\}^n$ be the hypercube equipped with the uniform probability measure $\mu_n$, and let $L_{p}(\{-1,1\}^{n})$ be the associated $L_{p}$ space for $1\leq p\leq \infty$. In the sequel, we will use the shorthand notation $[n]\coloneqq\{1,2,\cdots,n\}$. For each $j\in[n]$, the $j$-th influence of $f:\{-1,1\}^n\to\mathbb{R}$ is given by
$$\mathrm{Inf}_j(f)=\mu_n(\{x\in \{-1,1\}^n|\, f(x)\neq f(x^{\oplus j})\}),$$
where $x^{\oplus j}$ means flipping the $j$-th variable of $x$, i.e. for $x=(x_1,\cdots, x_n)\in \{-1,1\}^n$, 
$$x^{\oplus j}=(x_1,\cdots,x_{j-1},-x_j,x_{j+1},\cdots,x_n).$$
The total influence of $f$ is defined by $\mathrm{Inf}(f)=\sum_{j\in[n]}\mathrm{Inf}_j(f)$,  which is often used to measure the complexity of the function $f$. To illustrate the analytic property of the $j$-th influence of $f$, we now recall the $j$-th partial derivative of the $f:\{-1,1\}^{n}\to \mathbb{R}$ as follows:
\begin{equation}
\dc_{j}(f)(x)\coloneqq\frac{f(x)-f(x^{\oplus j})}{2},\quad x\in\{-1,1\}^{n}.
\end{equation}
In particular, for each Boolean function $f:\{-1,1\}^n\to\{-1,1\}$, it is easy to compute that
\begin{equation}\label{Inf-Dj}
\mathrm{Inf}_j(f)=\|\dc_{j}(f)\|_{L_p(\{-1,1\}^n)}^p,\quad \forall 1\leq p< \infty, j\in[n].
\end{equation}
Bounding the (total) influence of $f$ in terms of the variance of $f$ is one of the essential themes of Boolean analysis, which is closely related to specific types of functional inequalities on the hypercube. In particular, the Poincar\'e inequality on the hypercube can be stated as follows (see e.g. \cite[p.\,36]{ODo2014}):
\begin{equation}\label{Poincare inequality}
\var(f):=\|f-\mathbb{E}_{\mu_{n}}(f)\|_{L_2(\{-1,1\}^n)}^2\leq \sum_{j=1}^{n}\|\dc_{j}(f)\|^{2}_{L_{2}(\{-1,1\}^{n})},
\end{equation}
where $\mathbb{E}_{\mu_{n}}(f)$ is referred as the expectation of $f$ with respect to the uniform measure $\mu_{n}$. Hence, combining \eqref{Inf-Dj} with \eqref{Poincare inequality} leads to the following lower bound
\[
\max_{j\in[n]}\mathrm{Inf}_j(f)\geq \frac{1}{n},
\]
for each balanced Boolean function $f$, that is, a Boolean function with $\var(f)=1$.

However, in many aspects, the Poincar\'e inequality far from be sharp. In the remarkable paper \cite{KKL1988}, Kahn,  Kalai and Linial strengthened the Poincar\'e inequality in a fundamental way, which leads to the following inequality: there exists a universal constant $C>0$ such that 
\begin{equation}\label{KKL}
\max_{j\in[n]}\mathrm{Inf}_j(f)\geq C\frac{\log(n)}{n},
\end{equation}
for every balanced Boolean function $f$. More precisely,  Kahn, Kalai and Linial \cite{KKL1988} established the following functional inequality  elegantly.
\begin{theorem}[Kahn-Kalai-Linial]\label{classical KKL}
There exists a universal constant $C>0$ such that the following holds
\begin{equation}\label{functional KKL}
\var(f)\leq C\frac{\sum_{j=1}^{n}\|\dc_{j}(f)\|^{2}_{L_{2}(\{-1,1\}^{n})}}{\log\left(1/\max_{j\in [n]}\|\dc_{j}(f)\|^{2}_{L_{2}(\{-1,1\}^{n})}\right)},
\end{equation}
for every function $f:\{-1,1\}^{n}\to \mathbb{R}$.
\end{theorem} 
Due to the fundamental role in Boolean analysis, \eqref{KKL} (or, \eqref{functional KKL}) is now known as the KKL inequality, and we refer to \cite{ODo2014} to interesting applications of the inequality. One of significant improvements of the KKL inequality is the Talagrand ($L_{1}$-$L_{2}$-) influence inequality established in  \cite{Ta1994}. 
\begin{theorem}[Talagrand]
There exists a universal constant $C>0$ such that for each function $f:\{-1,1\}^{n}\to\mathbb{R}$ the following holds
\begin{equation}\label{Ta-L1L2}
\var(f)\leq C\sum_{j=1}^{n}\frac{\|\dc_{j}(f)\|^{2}_{L_{2}(\{-1,1\}^{n})}}{1+\log(\|\dc_{j}(f)\|_{L_{2}(\{-1,1\}^{n})}/\|\dc_{j}(f)\|_{L_{1}(\{-1,1\}^{n})})}.
\end{equation}
\end{theorem}
\noindent By \eqref{Inf-Dj}, it is clear that the Talagrand influence inequality \eqref{Ta-L1L2} implies the KKL inequality \eqref{KKL}. Since then, the KKL inequality, the Talagrand influence inequality, and their extensions become one of the fundamental tools in Boolean analysis, geometric functional analysis, computer science, and related fields. We refer interested readers to \cite{BKKKL1992,CL2012,EG2022,KKKMS2021,OW2013} for further information and the extensive bibliographies therein.

More recently, motivated by a conjecture of Talagrand \cite{Ta1997}, Eldan and Gorss \cite{EG2020} (see also \cite{EG2022}) applied stochastic analysis techniques to prove the following result known as the Eldan-Gross inequality.
\begin{theorem}[Eldan-Gross]\label{Eldan-Gross inequality}
There exists a universal constant $C>0$ such that the following inequality holds
\begin{equation}\label{EG-2022}
\var(f)\sqrt{\log\left(1+\frac{e}{\sum_{j=1}^{n}\|\dc_{j}(f)\|^{2}_{L_{1}(\{-1,1\}^{n})}}\right)}\leq C\left\|\left(\sum_{j=1}^{n}|\dc_{j}(f)|^{2}\right)^{1/2}\right\|_{L_{1}(\{-1,1\}^{n})},
\end{equation}
for every Boolean function $f:\{-1,1\}^{n}\to\{-1,1\}$.
\end{theorem}
Notably, the Eldan-Gross inequality (i.e., Theorem \ref{Eldan-Gross inequality}) unifies the KKL inequality (i.e., Theorem \ref{classical KKL}) and Talagrand's isoperimetric inequality \cite[Theorem 1.1]{Ta1993}, making Theorem \ref{Eldan-Gross inequality} into one of the most efficient tools in Boolean analysis. For some new proofs of \eqref{EG-2022}, we refer the interested reader to \cite{BIM2023,EKLM2022,IZ2024,Ros2020}.

It is worthwhile to mention that the original proofs of the KKL inequality and the Talagrand influence inequality rely on the hypercontractivity principle and the heat semigroup theory on hypercubes, while the proof of the Eldan-Gross inequality utilizes methods from stochastic analysis which is of different nature of the semigroup approach. Recently, the random restriction method has been viewed as a valuable tool for proving functional inequalities in the hypercube setting. Specifically, Kelman et al. \cite{KKKMS2021} applied this method to give a unified proof of the KKL inequality \eqref{KKL} and the Talagrand influence inequality \eqref{Ta-L1L2}, along with some extensions. Eldan et al. \cite{EG2022} reproved the Eldan-Gross inequality \eqref{EG-2022} and the Talagrand isoperimetric inequality (i.e., \cite[Theorem 1.1]{Ta1993}) via the random restriction technique. We refer the interested reader to \cite{KKOD2018,KKLMS2020,KKKMS2021,EKLM2022} for further recent developments of the Fourier random restriction method in Boolean analysis.

In the present paper, motivated by the quantum KKL conjecture and related problems, we aim to develop the random restriction method to the noncommutative (or quantum) settings and apply such method to establish some quantum analogies of \eqref{KKL}, \eqref{Ta-L1L2} and \eqref{EG-2022}.

In the quantum setting, the $n$-folds tensor product of $M_{2\times 2}(\mathbb{C})$ equipped with normalized trace, denoted by $(\M_{2^{n}},\tr)$ for short, is viewed as the noncommutative correspondence of $n$-dimensional hypercube $(\{-1,1\}^{n},\mu_{n})$. For each $j\in [n]$, let $\dc_{j}$ be the $j$-th partial derivative operator on $\M_{2^{n}}$ (see Sect. \ref{preliminaries} for the definition). Recall from \cite[Definition 3.1]{MO2010} that an element $T\in \M_{2^{n}}$ is said to be Boolean if $T$ is  self-adjoint and unitary, that is, $T^{*}=T$ and $T^{*}T=\mathbf{1}$. In \cite[Proposition 11.1]{MO2010}, Montanaro and Osborne derived a quantum analogy of the Talagrand influence inequality \eqref{Ta-L1L2} via the quantum hypercontractivity principle. However, due to some intrinsic differences between classical and quantum hypercubes, the quantum Talagrand influence inequality can not generally lead to the quantum KKL inequality. On the other hand, Montanaro and Osborne \cite[Proposition 11.5]{MO2010} applied some Fourier analysis techniques to derive a quantum KKL inequality for concrete quantum Boolean functions  $T$ fulfilling $\|\dc_{j}(T)\|_{L_{1}(\M_{2^{n}})}=\|\dc_{j}(T)\|^{2}_{L_{2}(\M_{2^{n}})}$ for each $j\in [n]$. Such observations lead them conjecture the following problem, known as the quantum KKL conjecture.
\begin{conjecture}[Quantum KKL conjecture]\label{QKKL}
There exists a universal constant $C>0$ such that for each $n\in \mathbb{N}$ and quantum Boolean function $T$ the following holds
\[
\max_{j\in [n]}\|\dc_{j}(T)\|^{2}_{L_{2}(\M_{2^{n}})}\geq \frac{C\var(T)\log(n)}{n}.
\]
\end{conjecture}

Recently, Rouz\'e, Wirth and Zhang \cite{RWZ2024} provided an alternative answer of the quantum KKL conjecture invoking the geometric influence of quantum Boolean functions. Notably, the main tools of their approach to the quantum KKL conjecture is the following quantum Talagrand-type inequality \cite[Theorem 3.6]{RWZ2024}, which was derived via the semigroup method.
\begin{theorem}[Rouz\'e-Wirth-Zhang]\label{RWZ-Talagrand}
There exists a universal constant $C>0$ such that for each $n\in \mathbb{N}$ and $1\leq p<2$ the following holds
\[
\var(T)\leq \left(\frac{C}{2-p}\right)\sum_{j=1}^{n}\frac{\|\dc_{j}(T)\|^{p}_{L_{p}(\M_{2^{n}})}(1+\|\dc_{j}(T)\|^{p}_{L_{p}(\M_{2^{n}})})}{1+\log^{+}(1/\|\dc_{j}(T)\|^{p}_{L_{p}(\M_{2^{n}})})},
\]
for every self-adjoint $T\in \M_{2^{n}}$ with $\|T\|_{L_{\infty}(\M_{2^{n}})}\leq 1$.
\end{theorem}
By Theorem \ref{RWZ-Talagrand}, one can easily derive the quantum KKL inequality invoking the $L_{p}$-influences in the following manner: There exists a universal constant $C>0$ such that for each $1\leq p<2$ and quantum Boolean function $T$, we have $\max_{j\in [n]}\|\dc_{j}(T)\|^{p}_{L_{p}(\M_{2^{n}})}\geq\frac{(2-p)C\log(n)}{n}$ for every $n\in \mathbb{N}$. More recently, Jiao, Luo and Zhou \cite{JLZ2025} investigate the quantum KKL conjecture in the canonical anti-commuting (CAR) algebra framework. More precisely, we established the  noncommutative Eldan-Gross inequality via the fermion oscillator semigroup theory and applied the noncommutative Eldan-Gross inequality to derive the following two types of noncommutative KKL inequalities. Let $\{Q_{j}\}_{j=1}^{n}$ be $n$-configuration observables and CAR algebra $\mathcal{A}_{car,n}$ be the $\ast$-algebra generated by $\{Q_{j}\}_{j=1}^{n}$.

\begin{theorem}[Jiao-Luo-Zhou]\label{JLZ thm 1}
There exists a universal constant $C>0$ such that, for each $\varepsilon\in(0,1)$ and each balanced Boolean function $T\in \mathcal{A}_{car,n}$, one of the following inequalities holds:
\begin{enumerate}[{\rm (i)}]
\item $\max_{j\in [n]}\|\dc_{j}(T)\|^{2}_{L_{2}(\mathcal{A}_{car,n})}\geq \frac{C\varepsilon\log(n)}{n}$;
\item $\max_{j\in [n]}\|\dc_{j}(T)\|_{L_{1}(\mathcal{A}_{car,n})}\geq \frac{C}{n^{(1+\varepsilon)/2}}$.
\end{enumerate}
\end{theorem}

To derive a noncommutative KKL inequality in the CAR algebra setting invoking $L_{2}$-influence, we introduce the index  for balanced Boolean function in $\mathcal{A}_{car,n}$ and proved the following result; see \cite[Theorem 6.7]{JLZ2025}.

\begin{theorem}\label{JLZ thm 2}
For each $n\in \mathbb{N}$ and every balanced Boolean function $T$ with $\mathrm{ind}(T)<2$, there exists a constant $C_{\mathrm{ind}(T)}>0$ (depending only on the index) such that
\[
\max_{j\in [n]}\|\dc_{j}(T)\|^{2}_{L_{2}(\mathcal{A}_{car,n})}\geq \frac{C_{\mathrm{ind}(T)}\log(n)}{n},
\]
where the definition of $\mathrm{ind}(T)$ will be given in Section \ref{quantum Eldan-Gross}.
\end{theorem}
\noindent Furthermore, it has been shown in \cite[Remark 6.5]{JLZ2025} that the CAR algebra counterpart of the KKL inequality (for $L_{2}$-influence) fails for general balanced Boolean functions. Precisely, let $T\coloneqq\frac{1}{\sqrt{n}}\sum_{j=1}^{n}Q_{j}$, and it is easy to see that $T$ is a balanced Boolean function in $\mathcal{A}_{car,n}$ such that $\|\dc_{j}(T)\|^{2}_{L_{2}(\mathcal{A}_{car,n})}=\frac{1}{n}$ for each $j\in [n]$, which disproves the KKL conjecture in the CAR algebra setting. Nevertheless, the quantum KKL conjecture of Montanaro and Osborne remains open.

In the present paper, we continue to explore quantum functional inequalities which are closely related to the quantum KKL conjecture. On the one hand, due to the fundamental role of the random restriction method in hypercubes, we develop the random restriction technique in the quantum setting. On the other hand, inspired by Kelman et al. \cite{KKKMS2021}, Rouz\'e et al \cite{RWZ2024} and Jiao et al. \cite{JLZ2025}, we will derive and recover all mentioned inequalities in the quantum setting via the quantum random restriction technique. The main results of the present paper are outlined as follows.

As the first main result of our random restriction technique, we establish
 the following dimension free quantum KKL inequality.
\begin{theorem}\label{main 23}
There exists a universal constant $K>0$ such that for each $1\leq p<2$ and each $T\in \M_{2^{n}}$ with $0\le T\leq 1$, the following holds
\begin{equation}\label{dim-free KKL}
\max_{j\in [n]}\|\dc_{j}(T)\|^{p}_{L_{p}(\M_{2^{n}})}\geq \frac{1}{4}\exp\left\{-\left(\frac{K}{2-p}\right)\cdot\frac{\sum_{j=1}^{n}\|\dc_{j}(T)\|^{p}_{L_{p}(\M_{2^{n}})}}{\var(T)}\right\}.
\end{equation}
\end{theorem}

Analogous to the approach presented in \cite{KKKMS2021}, we derive the following quantum KKL inequality (invoking $L_{p}$-influence) via Theorem \ref{main 23}, which was recently proved by Rouz\'e et al. \cite[Theorem 3.9]{RWZ2024}. And we show in Remark \ref{rem p2} that the following KKL-type inequality fails for $p=2$ even in the commutative case. Hence, Theorem \ref{TKKL} may be the best possible quantum KKL-type inequality for bounded elements.

\begin{theorem}\label{TKKL}
There exists a universal constant $C>0$ such that for every  $1\leq p<2$ and $T\in \M_{2^n}$ with $0\leq T \leq 1$, the following holds
\[
\max_{j\in[n]}\|\dc_{j}(T)\|^{p}_{L_{p}(\M_{2^{n}})}\geq C\frac{(2-p)\var(T)\log\left(n\right)}{n}.
\]
\end{theorem}

The second main ingredient of this paper consists of the following two quantum isoperimetric inequalities, which can be used to derive the quantum counterpart of KKL-type inequalities presented in \cite[Sect. 6]{JLZ2025}.

\begin{theorem}[Quantum Talagrand-type isoperimetric inequality]\label{quantum Ta}
There exists a universal constant $K>0$ such that for each projection $T\in \M_{2^{n}}$ the following holds
\begin{equation}\label{quantum Ta 1}
\var(T)\sqrt{\log\left(\frac{1}{\var(T)}\right)}\leq K\left\|\left(\sum_{j=1}^{n}|\dc_{j}(T)|^{2}\right)^{1/2}\right\|_{L_{1}(\M_{2^{n}})}.
\end{equation}
\end{theorem}
\noindent Combing Theorem \ref{quantum Ta} and estimations on the Fourier spectrum of projection $T\in \M_{2^{n}}$, we establish the quantum Eldan-Gross inequality as follows.
\begin{theorem}[Quantum Eldan-Gross inequality]\label{quantum E-G}
There exists a universal constant $K>0$ such that for each projection $T\in \M_{2^{n}}$ the following holds
\[
\var(T)\sqrt{\log\left(1+\frac{1}{\sum_{j=1}^{n}\|\dc_{j}(T)\|^{2}_{L_{1}(\M_{2^{n}})}}\right)}\leq K\left\|\left(\sum_{j=1}^{n}|\dc_{j}(T)|^{2}\right)^{1/2}\right\|_{L_{1}(\M_{2^{n}})}.
\]
\end{theorem}

The proofs of above theorems are provided in Section \ref{quantum KKL} and Section \ref{quantum Eldan-Gross}. As the applications of the quantum Eldan-Gross inequality, we establish quantum version of Theorem \ref{JLZ thm 1} and Theorem \ref{JLZ thm 2} in Section \ref{quantum Eldan-Gross}. Due to intrinsic differences between quantum and classical hypercubes, additional efforts must be made to overcome difficulties that arise from these differences when employing the random restriction method; see the proof of Proposition \ref{key lemma 1} for instance.

The rest of this paper is organized as follows. In Section \ref{preliminaries}, we present the necessary background and results in the quantum setting, including basic properties of the quantum Ornstein-Uhlenbeck semigroup, quantum hypercontractivity, and its equivalence to the quantum logarithmic Sobolev inequality. In Section \ref{quantum random-restriction}, we introduce the quantum Fourier random restriction technique and provide some fundamental estimations on the Fourier spectrum for elements in $\M_{2^{n}}$. This method and the associated estimates are frequently used to derive inequalities throughout the paper. From Section \ref{quantum KKL} to Section \ref{quantum Eldan-Gross}, we provide proofs of the previously presented theorems by combining the quantum random restriction technique with the quantum semigroup method. Additionally, we derive the corresponding quantum KKL-type inequality in the respective sections.

\begin{remark}\label{rem BGX}
After completing this work, we learned that Blecher, Gao and Xu \cite{BGX2024} developed a similar random restriction technique and applied it to investigate KKL inequality and high order extension of the Talagrand influence inequality in the quantum setting. Precisely, Theorem \ref{dim-free KKL} with $p=1$ is proved independently by Blecher et al. in \cite{BGX2024}.
\end{remark}

Throughout the paper, $n$ be a fixed positive integer and $[n]\coloneqq \{1,2,3\dots, n\}$. For a parameter $p$, we denote $K_{p}$ the positive constant depending only on the parameter $p$ (it may vary from line to line). We use the notation $A\approx_{p} B$ to stand that $K_{p}A\leq B\leq C_{p}A$ for some positive constants $K_{p}$ and $C_{p}$ (depending only on the parameter $p$), and we drop the subscript $p$ if the constants are universal. Notations $\mathbb{R}$ and $\mathbb{C}$ are the fields of real and complex numbers, respectively, and we let $(\M_{2}(\mathbb{C}),\tr)$ be the algebra of all $2\times 2$ complex matrices equipped with the normalized trace $\tr$.

\section{Preliminaries}\label{preliminaries}
In this section, we collect concepts and background that will be used throughout the paper.

\subsection{The quantum hypercube}
Denote by $(\mathbb{M}_{2}(\mathbb{C}), \mathrm{tr})$ the algebra of $2\times2$ complex matrices equipped with the \emph{normalized} trace $\mathrm{tr}$ and the unit $\mathbf{1}_{2}$ (i.e.,  the $2\times 2$ identity matrix). The quantum analogue of the hypercube $\{-1,1\}^n$ is $\mathbb{M}_{2}(\mathbb{C})^{\otimes n} \cong \mathbb{M}_{2^n}(\mathbb{C})$ equipped with the normalized trace $\tr_{n}\coloneqq \tr^{\otimes n}$ (simply denoted by $(\M_{2^{n}},\tr)$ if no confusion arise) and the unit $\mathbf{1}\coloneqq \mathbf{1}^{\otimes n}_{2}$. For every $1\leq p\leq\infty$, the noncommutative $L_{p}$ space generated by $\MP$, denoted by $L_{p}(\MP)$, is the space $\MP$ equipped with the norm
\[
\|T\|_{L_{p}}\coloneqq
\begin{cases}
\left(\tr(|T|^{p})\right)^{1/p},&\mbox{if }1\leq p<\infty,\\
\max_{j\in [n]}s_{j}(T),&\mbox{if }p=\infty,
\end{cases}
\]
where $\{s_{j}(T)\}_{j=1}^{n}$ is the set of singular values of $T$. The \emph{variance} of $T\in \M_{2^{n}}$ is defined by
\[
\var(T)\coloneqq \tr(|T|^{2})-\left|\tr(T)\right|^{2}=\left\|T-\tr(T)\right\|^{2}_{L_{2}}.
\]
Following \cite{MO2010}, we recall the concept of quantum Boolean functions in $\M_{2^{n}}$.
\begin{definition}[\cite{MO2010}]\label{def bqb}
An element $T\in \M_{2^{n}}$ is said to be a \emph{quantum Boolean function} if $T$ is self-adjoint (i.e., $T^{*}=T$) and unitary (i.e., $TT^{*}=T^{*}T=\mathbf{1}$). A quantum Boolean function $T\in \M_{2^{n}}$ is said to be balanced if $\tr(T)=0$.
\end{definition}

\begin{remark}\label{equivalence}
It is clear that the concept of quantum Boolean functions and projections are equivalent in the following sense. For a given quantum Boolean function $T\in \M_{2^{n}}$,  $S\coloneqq \frac{\mathbf{1}+T}{2}$ is a projection in $\M_{2^{n}}$, that is, $S^{*}=S=S^{2}$. Conversely, for a given projection $S\in \M_{2^{n}}$, it follows that $T\coloneqq 2S-\mathbf{1}$ is a quantum Boolean function in $\M_{2^{n}}$. Hence, we will ignore the difference between quantum Boolean functions and projections in $\M_{2^{n}}$.
\end{remark}

Using the Pauli matrices, we represent elements in $\M_{2^{n}}$ by their Fourier expansion, which is a quantum counterpart of the Walsh expansion for functions on $\{-1,1\}^{n}$. Recall that the Pauli matrices as follows:
\[
\sigma_{0}=
\begin{pmatrix}
1&0\\
0&1
\end{pmatrix},\quad
\sigma_{1}=
\begin{pmatrix}
	1 & 0\\
	0 & -1
\end{pmatrix},\quad
\sigma_{2}=
\begin{pmatrix}
	0 & 1\\
	1 & 0
\end{pmatrix},\quad 
\sigma_{3}=
\begin{pmatrix}
	0  & i\\
	-i & 0
\end{pmatrix}.
\]
For $\mathbf{s}=(s_i)_{i=1}^{n}\in\{0,1,2,3\}^{n}$, set 
\[
\sigma_{\mathbf{s}}=\sigma_{s_{1}}\otimes\cdots\otimes\sigma_{s_{n}}.
\]
Clearly,  $\{\sigma_{\mathbf{s}}\}_{\mathbf{s}\in\{0,1,2,,3\}^{n}}$ are quantum Boolean functions which forms an orthonormal basis in $L_{2}(\M_{2^{n}})$. Hence,  each $T\in \MP$ can be uniquely represented by
\[
T=\sum_{\mathbf{s}\in \{0,1,2,3\}^n}\widehat{T}(\mathbf{s})\sigma_{\mathbf{s}},
\]
where $\widehat{T}(\mathbf{s})$ is the Fourier coefficient defined by $\widehat{T}(\mathbf{s})=\tr(\sigma^{*}_{\mathbf{s}}T)$.

For each $j\in [n]$ and $\alpha\in \{1,2,3\}$, we let $e^{\alpha}_{j}\coloneqq(0,\dots,0,\alpha,0\dots,0)$ where $\alpha$ appears in the $j$-th position. For each $\s\in \{0,1,2,3\}^{n}$, $j\in [n]$ and $\alpha\in\{1,2,3\}$, define $\s\oplus e^{\alpha}_{j}$ (resp. $\s\ominus e^{\alpha}_{j}$) by
\[
\s\oplus e^{\alpha}_{j}\coloneqq(s_{1},\dots,s_{j-1},s_{j}+\alpha,s_{j},\dots,s_{n})
\]
\[
\left(\mbox{resp. }\s\ominus e^{\alpha}_{j}\coloneqq (s_{1},\dots,s_{j-1},s_{j}-\alpha,s_{j+1},\dots,s_{n})\right).
\]
 For any $d\in[n]$, the \emph{Rademacher projection} is defined by
\[
\mathrm{Rad_{\leq d}}(T)=\sum_{\substack{\s\in\{0,1,2,3\}^{n}\\|\supp(\s)|\leq d}}\widehat{T}(\s)\sigma_{\s},
\]
where $T=\sum_{\s\in\{0,1,2,3\}^{n}}\widehat{T}(\s)\sigma_{\s}$.

\subsection{Hypercontractivity, influence and the modified Log-Sobolev inequality}
Here we collect analytic tools such as the hypercontractivity of the quantum Ornstein-Uhlenbeck semigroup, basic properties of influences, and the curvature condition in quantum hypercubes. For $\s\in\{0,1,2,3\}^{n}$, we define $\supp(\s)\coloneqq\{j\in [n]: s_{j}\neq 0\}$ and $|\supp(\s)|$ stand for the number of non-zero $s_{j}$'s, that is, $|\supp(\s)|=\#\{j\in[n]:s_{j}\neq 0\}$. Let $L:\mathbb{M}_{2}\to\mathbb{M}_{2}$ defined by
\begin{equation}\label{Hu}
L(A)\coloneqq A-\mathrm{tr}(A)\mathbf{1}_{2},
\end{equation}
and
\begin{equation}\label{semigroup for tensor}
e^{-tL}(A)\coloneqq e^{-t}A+(1-e^{-t})\mathrm{tr}(A)\mathbf{1}_{2}.
\end{equation}
 In viewing of \eqref{Hu} and \eqref{semigroup for tensor}, we have 
\[
L(\sigma_{j})=
\begin{cases}
\sigma_{j},~\mbox{if}~j\not=0;\\
0,~\mbox{otherwise},
\end{cases}
~~~\mbox{and}~~~
e^{-tL}(\sigma_{j})=
\begin{cases}
e^{-t}(\sigma_{j}),~\mbox{if}~j\not=0;\\
\sigma_{0},~\mbox{otherwise}.
\end{cases}
\]
For each stabilizer operator $\sigma_{\mathbf{s}}$, let
\[
\mathcal{P}_{t}(\sigma_{\mathbf{s}})=e^{-tL}(\sigma_{s_{1}})\otimes\cdots\otimes e^{-tL}(\sigma_{s_{n}}).
\]
The infinitesimal generator of the semigroup $\{\mathcal{P}_{t}\}_{t\geq 0}$ is given by
\[
\begin{split}
\mathcal{L}(\sigma_{\mathbf{s}})\coloneqq\lim_{t\to 0}\frac{\sigma_{\mathbf{s}}-\left(e^{-tL}\right)^{\otimes n}(\sigma_{\mathbf{s}})}{t}=|\mathbf{s}|\sigma_{\mathbf{s}},
\end{split}
\]
for each stabilizer operator $\sigma_{\mathbf{s}}$. In the sequel, we denote $\{\mathcal{P}_{t}\}_{t\geq 0}$ by $\{e^{-t\mathcal{L}}\}_{t\geq 0}$ to emphasize the generator $\mathcal{L}$. It is clear that for each $T\in \M_{2^{n}}$ we have
\[
\La(T)=\sum_{\mathbf{s}\in\{0,1,2,3\}^{n}}|\mathbf{s}|\widehat{T}(\mathbf{s})\sigma_{\mathbf{s}}
\] 
and
\[
e^{-t\La}(T)=\sum_{\mathbf{s}\in\{0,1,2,3\}^{n}}e^{-t|\mathbf{s}|}\widehat{T}(\mathbf{s})\sigma_{\mathbf{s}}.
\]

For $j\in [n]$, the $j$-th \emph{partial differential operator} (or, quantum \emph{bit-flip map}) is defined by
\begin{equation}\label{jD}
\dc_{j}\coloneqq \mathbf{1}_{2}^{\otimes (j-1)}\otimes \left(\mathbf{1}_{2}-\tr\right)\otimes \mathbf{1}_{2}^{\otimes(n-j)}
\end{equation}
Thanks to the Fourier expansion of $T\in \MP$, we obtain the following explicit formula for partial differential operators. For each $j\in[n]$ and $T\in \MP$, we have
\begin{equation}\label{partial derivative}
\dc_{j}(T)=\sum_{\mathbf{s}\in\{0,1,2,3\}^{n}}\widehat{T}(\mathbf{s})\dc_{j}(\sigma_{\mathbf{s}})=\sum_{\substack{\mathbf{s}\in\{0,1,2,3\}^{n}\\s_{j}\not=0}}\widehat{T}(\mathbf{s})\sigma_{\mathbf{s}}.
\end{equation}
Moreover, it is easy to verify that $\{\dc_{j}\}_{j=1}^{n}$ are orthogonal projections on $L_{2}(\MP)$ such that $\La=\sum_{j=1}^{n}\dc_{j}$. By \cite{MO2010} (or, \cite[Corollary 2]{Ki2014}), the quantum Ornstein-Uhlenbeck semigroup $\{e^{-t\La}\}_{t\geq 0}$ fulfills the \emph{optimal} hypercontractivity as follows: for each $1<p\leq q<\infty$, we have
\begin{equation}\label{MO hypercontractivity}
\left\|e^{-t\La}\right\|_{L_{p}\to L_{q}}=1~\mbox{if and only if}~t\geq\frac{1}{2}\log\left(\frac{q-1}{p-1}\right).
\end{equation}
The equivalence between the hypercontractivity of semigroup and the logarithmic Sobolev inequality has been established by Gorss in his seminal paper \cite{Gr1975}. Hence, repeat the same treatments of Gorss, we can deduce the following ($L_{2}$-)logarithmic Sobolev inequality from \eqref{MO hypercontractivity} (the proof is also same to \cite[Theorem 5.3]{CL1993}).

\begin{lemma}[Log-Sobolev]\label{ls 1}
For each $T\in \MP$, we have
\[
2\sum_{j=1}^n\|\dc_{j}(T)\|^{2}_{L_{2}}\geq\tr\left[|T|^{2}\log\left(|T|^{2}\right)\right]-\|T\|^{2}_{L_{2}}\log\left(\|T\|^{2}_{L_{2}}\right).
\]
\end{lemma}
Motivated by the quantum KKL-type inequalities invoking $L_{p}$-influences with $1\leq p<\infty$, we derive the following ($L_{p}$-) \emph{modified logarithmic Sobolev inequality}. 

\begin{lemma}[Modified Log-Sobolev inequality]\label{mls}
Let $1\leq p<2$. Then, for each $T\in \MP$ with $|T| \leq 1$,  we have
\[
2\sum_{j=1}^n\|\dc_{j}(T)\|^{2}_{L_{2}}\geq -K_p\|T\|_{L_{2}}\|T\|^{\frac{p}{2}}_{L_{p}}-\|T\|^{2}_{L_{2}}\log\left(\|T\|^{2}_{L_{2}}\right),
\]
where $K_{p}=\frac{4}{(2-p)e}$.
\end{lemma}

\begin{proof}
By the Cauchy-Schwarz inequality, we obtain 
\begin{equation}\label{ls1}
\tr\left[-|T|^2\log(|T|^2)\right]\leq\|T\|_{L_{2}}\left[\tr\left(|T|^{2}\log^2\left(|T|^2\right)\right)\right]^{\frac{1}{2}}.
\end{equation}
Let $K_{p}\coloneqq\frac{4}{(2-p)e}$. Then, it is clear that
\[
K_p^2t^p\geq t^2\log^2\left(t^2\right),\quad \forall t\in [0,1].
\]
Hence, it follows from the functional calculus of $|T|$ that
\[
K_p^2|T|^p\geq |T|^2\log^2\left(|T|^2\right),\quad T\in \M_{2^{n}}.
\]
Therefore,
\begin{equation}\label{ls2}
\left[\tr\left(|T|^2\log^2(|T|^2)\right)\right]^{\frac{1}{2}}\leq K_p\|T\|_{L_{p}}^{\frac{p}{2}}.
\end{equation}
Combining \eqref{ls1}, \eqref{ls2} and the Log-Sobolev inequality (i.e., Lemma \ref{ls 1}), we get the desired result.
\end{proof}

To derive the isoperimetric inequality in Section \ref{quantum Eldan-Gross}, we need the following facts regarding as the curvature condition of the quantum Ornstein-Uhlenbeck semigroup $\{e^{-t\La}\}_{t\geq 0}$. 
\begin{proposition}\label{curvature condition}
Keep the notations as previous subsection. Then, for each $T\in \MP$, we have
\begin{enumerate}[{\rm (i)}]
\item $\La(T^{*}T)-\La(T)^{*}T-T^{*}\La(T)=-2\sum\limits_{j=1}^{n}(\dc_{j}(T))^{*}(\dc_{j}(T))$;
\item $(\dc_{j}e^{-t\La}(T))^{*}(\dc_{j}e^{-t\La}(T)\leq e^{-2t}e^{-t\La}\left((\dc_{j}(T))^{*}(\dc_{j}(T))\right)$, for each $j\in [n]$;
\item $\sum\limits_{j=1}^{n}(\dc_{j}e^{-t\La}(T))^{*}(\dc_{j}e^{-t\La}(T))\leq e^{-2t}e^{-t\La}\left(\sum\limits_{j=1}^{n}(\dc_{j}(T))^{*}(\dc_{j}(T))\right)$.
\end{enumerate}
\end{proposition}
\begin{proof}
The proof is analogous to the classical case, which can be verified via the Fourier expansion and the Gronwall-type inequality. Hence, we provide the proof of (ii) for the reader's convenience,  and leave the details of (i) and (iii) to the reader. For each $j\in [n]$ and $0\leq s\leq t$, we define
\[
\Lambda(s)\coloneqq e^{-(t-s)\La}\left[\left|\dc_{j} e^{-s\La}(T)\right|^{2}\right]=e^{-(t-s)\La}\left[\left(\dc_{j} e^{-s\La}(T)\right)^{*}\left(\dc_{j} e^{-s\La}(T)\right)\right].
\]
Differentiating $\Lambda(s)$ and applying (i) we obtain that
\begin{equation}\label{differential inequality 1}
\begin{split}
\Lambda^{\prime}(s)&=\La\left[e^{-(t-s)\La}\left(\left(\dc_{j}e^{-s\La}(T)\right)^{*}\left(\dc_{j}e^{-s\La}(T)\right)\right)\right]\\
&\quad-e^{-(t-s)\La}\left[\left(\La\dc_{j}e^{-s\La}(T)\right)^{*}\left(\dc_{j}e^{-s\La}(T)\right)\right]\\
&\quad-e^{-(t-s)\La}\left[\left(\dc_{j}e^{-s\La}(T)\right)^{*}\left(\La\dc_{j}e^{-s\La}(T)\right)\right]\\
&=e^{-(t-s)\La}\Big{[}\La\left(\left(\dc_{j}e^{-s\La}(T)\right)^{*}\left(\dc_{j}e^{-s\La}(T)\right)\right)\\
&\quad-\left(\La\dc_{j}e^{-s\La}(T)\right)^{*}\left(\dc_{j}e^{-s\La}(T)\right)\\
&\quad-\left(\dc_{j}e^{-s\La}(T)\right)^{*}\left(\La\dc_{j}e^{-s\La}(T)\right)\Big{]}\\
&\leq -2e^{-(t-s)\La}\left[\left(\dc_{j}e^{-s\La}(T)\right)^{*}\left(\dc_{j}e^{-s\La}(T)\right)\right]\\
&=-2\Lambda(s).
\end{split}
\end{equation}
Define $F(s)\coloneqq e^{2s}\Lambda(s)$ and note that \eqref{differential inequality 1} entails $F^{\prime}(s)\leq 0$ for all $0\leq s\leq t$. Therefore,
\[
e^{2t}\Lambda(t)-\Lambda(0)=F(t)-F(0)=\int^{t}_{0}F^{\prime}(s)~ds\leq 0.
\]
Rearranging the inequality yields that $\Lambda(t)\leq e^{-2t}\Lambda(0)$, that is,
\[
\left|\dc_{j}e^{-t\La}(T)\right|^{2}\leq e^{-2t}e^{-t\La}\left[\left|\dc_{j}(T)\right|^{2}\right].
\]
\end{proof}

We conclude this subsection with the following well-known Paley-Zygmund inequality, and we include the proof for the reader's convenience.

\begin{lemma}[Paley-Zygmund inequality]\label{Paley-Zygmud}
For each positive $T\in \M_{2^{n}}$, we have
\begin{equation}\label{eq:Paley-Zygmud}
\tr\left[\mathds{1}_{[\delta\|T\|_{L_{1}},\infty)}(T)\right]\geq {(1-\delta)^2}\frac{\|T\|^{2}_{L_{1}}}{\|T\|^{2}_{L_{2}}},\quad 0<\delta<1.
\end{equation}
\end{lemma}
\begin{proof}
Since the desired inequality only invokes one positive element, it follows from the spectral theory that the inequality is essentially the classical Paley-Zygmund inequality. For positive $T\in \M_{2^{n}}$, we have
\begin{equation}\label{P-Z inequality}
\begin{split}
\left\|T\right\|_{L_{1}}&=\tr\left[\mathds{1}_{[\delta\|T\|_{L_{1}},\infty)}(T)\cdot T\right]+\tr\left[\mathds{1}_{[0,\delta\|T\|_{L_{1}})}(T)\cdot T\right]\\
&\leq \tr\left[\mathds{1}_{[\delta\|T\|_{L_{1}},\infty)}(T)\right]^{1/2}\|T\|_{L_{2}}+\delta\|T\|_{L_{1}}
\end{split}
\end{equation}
where we used the the Cauchy-Schwarz inequality. Rearranging \eqref{P-Z inequality} yields the desired inequality.
\end{proof}

\subsection{$L_{p}$-influences and related basic properties}
For $j\in [n]$, $1\leq p<\infty$ and $T\in \MP$, denote the $j$-th $L_p$-influence of $T$ by
\[
\mathrm{Inf}_j^p(T)\coloneqq\|\dc_{j}(T)\|^{p}_{L_{p}},
\] 
and the total $L_p$-influence of $T$ by 
\[
\mathrm{Inf}^{p}\coloneqq \sum_{j=1}^{n}\|\dc_{j}(T)\|^{p}_{L_{p}}.
\]
The $L_1$-influence is usually called the \emph{geometric influence} in some literature. For $p=2$, we will simply denote the $j$-th $L_{2}$-influence and the total $L_2$-influence of $T$ by $\mathrm{Inf}_{j}(T)$ and $\mathrm{Inf}(T)$, respectively. Hence, by \eqref{partial derivative}, it follows that for each $T\in \M_{2^{n}}$, we have
\[
\mathrm{Inf}_{j}(T)=\sum_{\substack{\s\in\{0,1,2,3\}^{n}\\s_{j}\neq 0}}|\widehat{T}(\mathbf{s})|^{2},
\]
and
\[
\mathrm{Inf}(T)=\sum_{j\in [n]}\mathrm{Inf}_{j}(T)=\sum_{\s\in\{0,1,2,3\}^{n}}|\supp(\s)|\widehat{T}(\s)^{2}.
\]
The following elementary facts can be deduced from the contraction of conditional expectations and the noncommutative H\"older inequality.

\begin{proposition}\label{dlili}
For $1\leq p\leq 2$ and $T\in \MP$ with $\|T\|_{L_{\infty}}\leq 1$ we have
\begin{enumerate}[{\rm (i)}]
\item for each $j\in [n]$, we have $\|\dc_{j}(T)\|_{L_{\infty}}\leq 1$,
\item for each $j\in [n]$, we have $\mathrm{Inf}_{j}(T)\leq \mathrm{Inf}^{p}_{j}(T)$.
\end{enumerate} 
\end{proposition}
\begin{proof}
(i) For each $j\in [n]$, we define $S:\M_{2}\to \M_{2}$ by
\[
S(T)\coloneqq
\begin{pmatrix}
0&-i\\
i&0
\end{pmatrix}
T^{t}
\begin{pmatrix}
0&-i\\
i&0
\end{pmatrix},
\]
where $T^{t}$ stands for the transpose of $T$. It is clear that
\[
S(\sigma_{0})=\sigma_{0} \quad\mbox{and}\quad S(\sigma_{l})=-\sigma_{l}
\]
for $l\in\{1,2,3\}$. We further define $S_{j}\coloneqq \mathbf{1}^{\otimes (j-1)}_{2}\otimes S\otimes \mathbf{1}^{\otimes (n-j)}_{2}$. Hence, it is clear that
\[
\dc_{j}(T)=\frac{1}{2}\left(T-S_{j}(T)\right),\quad\mbox{for } T\in \M_{2^{n}}.
\]
Since the matrix
$\begin{pmatrix}
0&-i\\
i&0
\end{pmatrix}$ is unitary and the norm $\|\cdot\|_{L_{p}}$ is unitary invariant for every $p\in(0,\infty]$, it follows that $\left\|S_{j}(T)\right\|_{L_{\infty}}=\|T\|_{L_{\infty}}$. Therefore, we have
\[
\|\dc_{j}(T)\|_{L_{\infty}}=\frac{1}{2}\left\|T-S_{j}(T)\right\|_{L_{\infty}}\leq \|T\|_{L_{\infty}},
\]
which proves the first claim.

(ii) For each $j\in[n]$, we have
\[
\mathrm{Inf}_{j}(T)=\|\dc_{j}(T)\|^{2}_{L_{2}}\leq \||\dc_{j}(T)|^{p}\|_{L_{1}}\cdot\||\dc_{j}(T)|^{2-p}\|_{L_{\infty}}\leq \|\dc_{j}(T)\|^{p}_{L_{p}},
\] 
where we used $1\leq p<2$ and $\|\dc_{j}(T)\|_{L_{\infty}}\leq 1$.
\end{proof}

\section{The Noncommutative Random Restrictions and Related Estimates}\label{quantum random-restriction}
In this section, motivated by \cite{KKKMS2021}, we introduce a noncommutative random restriction technique, which is one of the efficient toolkits of establishing functional inequalities in the quantum hypercube. 

For each subset $J=\{j_1,\cdots, j_k\}\subseteq [n]$, we order it in the increasing order, that is, $j_1<j_2<\cdots<j_k$, and let $\M_{J}$ be the $*$-sub-algebra of $\M_{2^{n}}$ such that $\M_2$ only appears in the $j_{i}$-th position for $1\leq i\leq k$. Hence, there exists a conditional expectation $\mathcal{E}_{\M_{J}}$ from $\M_{2^{n}}$ onto $\M_{J}$. More precisely, $\mathcal{E}_{\M_{J}}$ has the following explicit formula.
\begin{proposition}[Conditional Expectation]\label{conditionalexpectation}
For each $J\subseteq [n]$, we have
 \[
\mathcal{E}_{\M_J}(T)\coloneqq\sum_{\supp(\mathbf{s})\subseteq J}\widehat{T}(\mathbf{s})\sigma_{\mathbf{s}},\quad \forall~T\in \M_{2^{n}}.
\]
\end{proposition}

We now introduce the \emph{random restriction operator} as follows.

\begin{definition}[Restrictions]\label{restriction}
For each $J\subseteq[n]$ and $j\in [n]$, let $J^{c}\coloneqq [n]\setminus J$ and define the \emph{restriction operator}
$R_j^J:\MP\rightarrow \MP$ by setting
\[
R_j^J(T)\coloneqq
\begin{cases}
 \Ex_{\M_{J^c\cup \{j\}}}(\dc_{j}(T)),& j\in J\\
0,&j\in J^{c}.
\end{cases}
\]
\end{definition}
The following essential property regarding the explicit formula for the restriction operator acting on an element is easily deduced from the formula of conditional expectations (i.e., Proposition \ref{conditionalexpectation}) and the Fourier expansion of partial derivatives \eqref{partial derivative}. The proof is left to the interested reader.
\begin{lemma}\label{lem:rr1}
Let $J\subseteq [n]$ and $T\in \MP$. For each $j\in [n]$, we have
\[
R_j^J(T)=\sum_{\substack{\supp(\s)\subseteq J^{c}\\ \alpha\in\{1,2,3\}}}\widehat{T}(\s\oplus e_j^{\alpha})\sigma_{\s\oplus e_j^{\alpha}}.
\]
\end{lemma}
Applying Lemma \ref{lem:rr1}, we obtain the following corollary.
\begin{corollary}\label{cor:Ik1}
Let $T\in \MP$ and a fixed subset  $J\subseteq [n]$.
\begin{enumerate}[{\rm (i)}]
\item For each $j\in J$, we have
\[
\sum_{k=1}^{n}\|\dc_k(R_j^J(T))\|_{L_2}^2=\sum_{k\in J^c}\|\dc_k(R_j^J(T))\|_{L_2}^2+\|\dc_j(R_j^J(T))\|_{L_2}^2.
\]
\item We have that
\[
\sum_{j\in J}\|R_j^J(T)\|_{L_2}^2\leq \mathrm{var}(T).
\]
\item For each $k\in J^c$, we have
\[
\sum_{j\in J}\|\dc_k(R_j^J(T))\|_{L_2}^2\leq \left\|\dc_{k}(T)\right\|^{2}_{L_{2}}.
\]
\end{enumerate}
\end{corollary}

\begin{proof}
We only show item (iii), and leave the easy verification of (i) and (ii) to the reader. By Lemma \ref{lem:rr1}, we have
\begin{align*}
\sum_{j\in J}\|\dc_k(R_j^J(T))\|_{L_2}^2&=
\sum_{j\in J} \tr\left|\sum_{\substack{\supp(\s)\subseteq J^{c},\\s_k\neq0,\alpha\in\{1,2,3\}}}\widehat{T}(\s\oplus e_j^{\alpha})\sigma_{\s\oplus e_j^{\alpha}} \right|^2\\
&\leq \sum_{\supp(\widetilde{\s})\subseteq J} \tr\left| \sum_{\substack{s_{k}\neq 0\\\supp(\s)\subseteq J^{c}}}\widehat{T}(\s\oplus \widetilde{\s})\sigma_{\s\oplus \widetilde{\s}} \right|^2\\
&= \sum_{\supp(\widetilde{\s})\subseteq J} \sum_{\substack{s_k\neq 0\\\supp(\s)\subseteq J^{c}}} \left|\widehat{T}(\s\oplus \widetilde{\s})\right|^2\\
&=\sum_{s_k\neq0} \left|\widehat{T}(\s)\right|^2=\|\dc_k(T)\|_{L_2}^2.
\end{align*}
\end{proof}

Several basic properties regarding the $L_{p}$-norm of the restriction operator and its relation to the influence are collected as follows. 

\begin{lemma}\label{prrr}
Let $J\subseteq [n]$ and $T\in \MP$ with $\|T\|_{L_{\infty}}\leq 1$.
\begin{enumerate}[{\rm (i)}]
\item For each $j\in [n]$, we have $\|R_j^J(T)\|_{L_p}^p\leq \|\dc_{j}(T)\|^p_{L_p}$,  $1\leq p\leq \infty.$ 
\item For each $j\in [n]$, we have  $\|R_j^J(T)\|_{L_2}^2\leq\|R_j^J(T)\|_{L_p}^p$,  $1\leq p\leq 2$. 
\item We have $\sum_{j\in J}\mathrm{Inf}\left(R_j^J(T)\right)\leq \var(T)+\mathrm{Inf}(T)$.
\end{enumerate}
\end{lemma}
\begin{proof}
(i) It follows from the definition of restriction operator that
\begin{equation}\label{RDj}
\|R_j^J(T)\|_{L_p}=\|\mathbb{E}_{\M_{J^c\cup\{j\}}}(\dc_{j}(T))\|_{L_p}\leq\|\dc_{j}(T)\|_{L_p}.
\end{equation}
 For (ii), it suffices to note that
\begin{align*}
\|R_j^J(T)\|^2_{L_2}&=\||R_j^J(T)|^p|R_j^J(T)|^{2-p}\|_{L_1}\\
\leq&\||R_j^J(T)|^p\|_{L_1}\cdot\||R_j^J(T)|^{2-p}\|_{L_\infty}\leq \|R_j^J(T)\|^p_{L_p},
\end{align*}
where the last inequality follows from  \eqref{RDj} and Proposition \ref{dlili} (i).

(iii)  By items (i) and (iii) in Corollary \ref{cor:Ik1}, we have
\begin{align*}
\mathrm{Inf}(T)&=\sum_{k\in[n]}\|\dc_k(T)\|_{L_2}^2\geq\sum_{k\in J^c}\|\dc_k(T)\|_{L_2}^2\geq \sum_{k\in{J^c}}\sum_{j\in J}\|\dc_k(R_j^J(T))\|_{L_2}^2\\
&=\sum_{j\in J}\left(\sum_{k\in{J^c}}\|\dc_k(R_j^J(T))\|_{L_2}^2\right)=\sum_{j\in J}\left(\sum_{k=1}^{n}\|\dc_k(R_j^J(T))\|_{L_2}^2-\|\dc_j(R_j^J(T))\|_{L_2}^2\right).
\end{align*}
Applying Proposition \ref{dlili} (i) again, we get
\begin{align*}
\mathrm{Inf}(T)&\geq\sum_{j\in J}\left(\sum_{k=1}^{n}\|\dc_k(R_j^J(T))\|_{L_2}^2-\|R_j^J(T)\|_{L_2}^2\right)\\
&=\sum_{j\in J}\mathrm{Inf}(R_j^J(T))-\sum_{j\in J}\|R_j^J(T)\|_{L_2}^2\geq-\var(T)+\sum_{j\in J}\mathrm{Inf}(R_j^J(T)),
\end{align*}
where the last inequality follows from Corollary \ref{cor:Ik1} (ii).
\end{proof}
To investigate the Fourier spectrum of a given $T\in \M_{2^{n}}$, we introduce the following notations, which are inspired from their Boolean counterparts.
\begin{definition}
For each $T\in \MP$ and for $d\in [n]$, define
\[
W_{=d}(T)\coloneqq\sum_{|\supp(\s)|=d}\widehat{T}\left(\s\right)^{2},
\]
\[
W_{\geq d}(T)\coloneqq\sum_{|\supp(\s)|\geq d}\widehat{T}(\s)^{2},
\]
and
\[
W_{\approx d}(T)\coloneqq\sum_{d\leq|\supp(\s)|<2d}\widehat{T}\left(\s\right)^{2}.
\]
\end{definition}

\begin{remark}
It is clear that every random set $J\subseteq [n]$, formed by each point selected with probability $\delta$, corresponds to a vector in $(\{0,1\}^{n},\mu_{\delta})$, where
\[
\mu_{\delta}(\{x\})=\delta^{\sum_{j=1}^{n}x_{j}}\left(1-\delta\right)^{n-\sum_{j=1}^{n}x_{j}},\quad x=(x_{j})_{j=1}^{n}\in \{0,1\}^{n}.
\]
If there is no confusion arises, we will simply write $J\in(\{0,1\}^{n},\mu_{\delta})$ for a random set $J$ (with selecting probability $\delta$). 
\end{remark}

\begin{lemma}\label{zrr}
Let $d\in \Z_+$ and $T\in \MP$ with $\|T\|_{L_{\infty}}\leq1$. Then  
\[
\E_J\left[\sum_{j\in J}\|R_j^J(T)\|_{L_2}^2\right]\geq \frac{1}{8}\wyd[T],
\]
where $\mathbb{E}_{J}$ is the expectation taking with respect to the random subset $J$ with selecting probability $\frac{1}{d}$. Hence, there exists $J_0\subseteq[n]$ such that $\sum_{j\in J_0}\|R_j^J(T)\|_{L_2}^2\geq \frac{1}{8}W_{\approx d}T$.
\end{lemma}

\begin{proof}
Using Lemma \ref{lem:rr1} and the orthogonality of $\{\sigma_{\s}\}_{\s\in \{0,1,2,3\}^n}$, we have
\begin{equation}\label{spectrum 1}
\begin{split}
\mathbb{E}_J\left[\sum_{j\in J}\|R_j^J(T)\|_{L_2}^2\right]
&=\mathbb{E}_J \left[\sum_{j\in J}\sum_{\substack{\supp(\s)\subseteq J^{c}\\\alpha\in\{1,2,3\}}}|\widehat{T}(\s\oplus e_j^{\alpha})|^2\right]\\
&=\mathbb{E}_{J}\left[\sum_{\s\in\zfourn}|\widehat{T}(\s)|^2\mathbbm{1}_{\{J:|\supp(\s)\cap J|=1\}}\right]\\
&=\sum_{\s\in\{0,1,2,3\}^n}|\widehat{T}(\s)|^2\mu_{\frac{1}{d}}{\{J:|\supp(\s)\cap J|=1\}}\\
&\geq \sum_{d\leq |\supp(\s)|<2d}|\widehat{T}(\s)|^2\mu_{\frac{1}{d}}{\{J:|\supp(\s)\cap J|=1\}}.
\end{split}
\end{equation}
For the case $d=1$ with $d\leq |\supp(\s)|<2d$, we have  $|\supp(\s)|=1$, and hence,
\begin{equation}\label{verification 1}
\mu_{\frac{1}{d}}{\{J:|\supp(\s)\cap J|=1\}}=\frac{1}{d}=1.
\end{equation}
For the case $d>1$ with $d\leq |\supp(\s)|<2d$, we have
\begin{equation}\label{verification 2}
\begin{split}
\mu_{\frac{1}{d}}{\{J:|\supp(\s)\cap J|=1\}}&=\left(1-\frac{1}{d}\right)^{|\supp(\s)|-1}\left(\frac{|\supp(\s)|}{d}\right)\\
&\geq \inf_{d>1}\left(1-\frac{1}{d}\right)^{2d-1}\geq\frac{1}{8}.
\end{split}
\end{equation}
Substituting \eqref{verification 1} and \eqref{verification 2} to \eqref{spectrum 1} yields the desired result.
\end{proof}

\begin{lemma}
Let $T\in \MP$ with $0\leq T\leq 1$. Then, for $1\leq p<2$ and $J\subseteq[n]$, we have 
\begin{align*}
\mathrm{Inf}(T)+\var(T)&\geq \frac{1}{2}\sum_{j\in J}\|R_j^J(T)\|_{L_2}^2\log\left(\frac{1}{\max\limits_{j\in J} \|\dc_j(T)\|_{L_p}^p}\right)\\
&\quad-\frac{K_p}{2}\sqrt{\sum_{j\in J}\|R_j^J(T)\|_{L_2}^2}\sqrt{\mathrm{Inf}^p(T)},
\end{align*}
where $K_p=\frac{4}{(2-p)e}$.
\end{lemma}

\begin{proof}
For each $J\subseteq[n]$, by Lemma \ref{prrr}(\rm{iii}) and Lemma \ref{mls}, we have
\begin{align*}
&\mathrm{Inf}(T)+\var(T)\\
\geq&\frac{1}{2}\sum_{j\in J}\left(\|R_j^J(T)\|_{L_2}^2\log\left(\frac{1}{\|R_j^J(T)\|_{L_2}^2}\right)-K_p\sqrt{\|R_j^J(T)\|_{L_2}^2}\sqrt{\|R_j^J(T)\|_{L_p}^p}\right)\\
\geq&\frac{1}{2}\sum_{j\in J}\|R_j^J(T)\|_{L_2}^2\log\left(\frac{1}{\|R_j^J(T)\|_{L_2}^2}\right)-\frac{K_p}{2}\sqrt{\sum_{j\in J}\|R_j^J(T)\|_{L_2}^2}\sqrt{\sum_{j\in J}\|R_j^J(T)\|_{L_p}^p},
\end{align*}
where the last inequality is due to the Cauchy-Schwarz inequality. The desired assertion now follows from Lemma \ref{prrr} (i) and (ii).
\end{proof}

Choosing $J=J_{0}$ as in Lemma \ref{zrr}, we can relate the term $\sum_{j\in J}\left\|R^{J}_{j}(T)\right\|^{2}_{L_{2}}$ with $W_{\approx d}(T)$ and  obtain the following corollary.
\begin{corollary}\label{comlemma}
 Let $1\leq p<2$ and $T\in \MP$ with $0\leq T\leq 1$. Then the following holds
\[
\mathrm{Inf}(T)+\var(T)\geq \frac{1}{16}\log\left(\frac{1}{\max_{j\in[n]}\|\dc_j(T)\|_{L_p}^p}\right)W_{\approx d}T-\frac{K_p}{16}\sqrt{\mathrm{Inf}^p(T)}\sqrt{W_{\approx d}(T)},
\]
where $K_p=\frac{4}{(2-p)e}$.
\end{corollary}

\section{Proofs of Theorem \ref{dim-free KKL} and Theorem \ref{TKKL} }\label{quantum KKL}

The proof of Theorem \ref{main 23} is a little bit lengthy, which relies on sequence lemmas regarding decomposition of the Fourier spectrum. Hence, we will postpone the proof of Theorem \ref{main 23} and show how it can be used to deduce Theorem \ref{TKKL} at first. Our method presented below is essential inspired by the approach in \cite{KKKMS2021}.


\begin{proof}[Proof of Theorem \ref{TKKL}]
Assume $\mathrm{Inf}^p(T)\geq \frac{\var(T)\log(n)}{192e^{2}K_{p}}$ ($K_{p}$ is the same as in Theorem \ref{mls}). Noting that $\mathrm{Inf}^p(T)=\sum_{j=1}^{n}\|\dc_j(T)\|_{L_p}^p$, it follows that
\[
\max_{j\in[n]}\|\dc_j(T)\|_{L_p}^p\geq \frac{\var(T)\log(n)}{192e^{2}K_{p}n}.
\]
 If $\mathrm{Inf}^p(T)<\frac{\log(n)\var(T)}{192e^{2}K_{p}}$, then, by Theorem \ref{main 23}, we have
 \[
 \max_{j\in[n]}\|\dc_j(T)\|_{L_p}^p\geq\frac{1}{4\sqrt{n}}\geq\frac{\log(n)}{4n}\geq\frac{\var(T)\log(n)}{4n}\geq\frac{\var(T)\log(n)}{4n}.
 \]
This completes the proof.
\end{proof}

To prove Theorem \ref{main 23}, we need a sequence of technical lemmas, which are necessary estimations on the Fourier spectrum. We now begin by introducing the operator $\delta^{\La}$ on $\M_{2^{n}}$ via functional calculus of the non-negative generator $\La$ for $\delta\in[0,1]$. More precisely, by the Fourier expansion, we have
\[
\delta^{\La}(T)=\sum_{\s\in\{0,1,2,3\}^{n}}\delta^{|\supp(\s)|}\widehat{T}(\s)\sigma_{\s},
\]
where $T=\sum_{\s\in\{0,1,2,3\}^{n}}\widehat{T}(\s)\sigma_{\s}\in \M_{2^{n}}$. The operator $\delta^{\La}$ is one of the key ingredients of decomposing the Fourier spectrum of $T$. In the next proposition, we represent the operator $\delta^{\La}$ in terms of conditional expectation.


\begin{proposition}\label{tav}
For each $T\in \MP$, we have $\delta^{\mathcal{L}}(T)=\E_J[\mathcal{E}_{\M_{J}}(T)]$, where $J$ is a random set in $[n]$ corresponds to vector in $(\{0,1\}^{n},\mu_{\delta})$.
\end{proposition}

\begin{proof}
It is clear that
\begin{align*}
 \E_J(\mathcal{E}_{\M_{J}}(T))&=\E_{J}\left(\sum_{\supp(\s)\subseteq J}\widehat{T}(\s)\sigma_\s\right)\\
 &=\sum_{\s\in\zfourn}\mu_{\delta}\{J\in\{0,1\}^{n} :\supp(\s)\subseteq J\}\widehat{T}(\s)\sigma_\s \\
 &=\sum_{\s\in\zfourn}\delta^{|\supp(\s)|}\widehat{T}(\s) \sigma_\s.
\end{align*}
\end{proof}

\begin{lemma}\label{lehd}
Let $1\leq p<2$, $T\in \MP$ with $0\leq T\leq1$, and let $\mathbb{D}=\{2^k\}_{k\in \mathbb{Z}_+}$. For each $d\in \mathbb{D}$, we set $H_d(T)=(1-\frac{1}{2d})^{\mathcal{L}}(T)-(1-\frac{1}{d})^{\mathcal{L}}(T)$. Then we have
\begin{enumerate}[{\rm(i)}]
\item for each $d\in\mathbb{D}$, $|H_d(T)|\leq 1$; 
\item  for each $d\in\mathbb{D}$ and $j\in [n]$, $\|\dc_j(H_d(T))\|^p_p\leq2^p\|\dc_j(T)\|_{L_p}^p$;
\item for each $\s\in \{0,1,2,3\}^n$ with $d\leq |\supp(\s)|<2d$, we have $|\widehat{T}\left(\s\right)|\geq|\widehat{H}_d\left(\s\right)|\geq \frac{1}{4}|\widehat{T}\left(\s\right)|$;
\item $\sum_{d\in \mathbb{D}}\var(H_d(T))=\sum_{d\in \mathbb{D}}\|H_d(T)\|^2_{L_2}\leq \mathrm{var}\left(T\right)$ and 
$$\sum_{d\in \mathbb{D}} \mathrm{Inf}(H_d(T))\leq \mathrm{Inf}(T).$$
\end{enumerate}
\end{lemma}

\begin{proof}

(i) Note that by Proposition \ref{tav}, for each $\delta\in[0,1]$, $\delta^{\mathcal{L}} (T)=\mathbb{E}_J[\mathcal{E}_{\M_{J}}T]$ which implies $0\leq \delta^{\mathcal{L}} (T) \leq 1$. It follows that \[|H_d(T)|=\left|(1-\frac{1}{2d})^{\mathcal{L}}(T)-(1-\frac{1}{d})^{\mathcal{L}}(T)\right|\leq1.\]

(ii) Using Proposition \ref{tav} and the Jensen inequality, we have 
\begin{align*}
\|\dc_j(\delta^{\mathcal{L}}(T))\|_{L_p}^p&=\|\dc_{j}(\delta^{\mathcal{L}}(T))\|_{L_p}^p=\| \dc_{j}(\mathbb{E}_J(\mathcal{E}_{\M_{J}}(T)))\|_{L_p}^p\\
&=\|\mathbb{E}_{J}(\mathcal{E}_{\M_{J}}(\dc_{j}(T)))\|_{L_p}^p\\
&\leq \mathbb{E}_J\| \dc_{j}(T)\|_{L_p}^p=\| \dc_{j}(T)\|_{L_p}^p.
\end{align*}
Moreover, we have
\[
\|\dc_j(H_{d}(T))\|_{L_p}^p=\left\|\dc_j\left((1-\frac{1}{2d})^{\mathcal{L}}(T)-(1-\frac{1}{d})^{\mathcal{L}}(T)\right)\right\|_{L_p}^p\leq2^p\|\dc_j(T)\|_{L_p}^{p}.
\]
(iii) By the definition of $H_{d}$, we have
\[
|\widehat{H}_{d}\left(\s\right)|=\left(\left(1-\frac{1}{2d}\right)^{|\supp(\s)|}-\left(1-\frac{1}{d}\right)^{|\supp(\s)|}\right)|\widehat{T}(\s)|.
\]
It is clear that $|\widehat{H}_d\left(\s\right)|\leq |\widehat{T}(\s)|$ for each $\s\in \{0,1,2,3\}^n$ with $d\leq |\supp(\s)|<2d$. 
On the other hand side, we have
\[
\left(1-\frac{1}{2d}\right)^{|\supp(\s)|}-\left(1-\frac{1}{d}\right)^{|\supp(\s)|}\geq \left(1-\frac{1}{2d}\right)^{2d}\geq\frac{1}{4},
\]
where we used $d<|\supp(\s)|<2d$ in the last inequality.

(iv) By the definition of $H_d$ again, we have the following
\[
\sum_{d\in \mathbb{D}}\|H_d(T)\|^2_{L_2}=\sum_{d\in \mathbb{D}} \sum_{\s\in\zfourn} \left(\left(1-\frac{1}{2d}\right)^{|\supp(\s)|}-\left(1-\frac{1}{d}\right)^{|\supp(\s)|}\right)^2|\widehat{T}\left(\s\right)|^2,
\]
and
\begin{align*}
&\sum_{d\in \mathbb{D}}\mathrm{Inf}(H_d(T))\\
=&\sum_{d\in \mathbb{D}} \sum_{\s\in\zfourn} |\supp(\s)| \left(\left(1-\frac{1}{2d}\right)^{|\supp(\s)|}-\left(1-\frac{1}{d}\right)^{|\supp(\s)|}\right)^2|\widehat{T}\left(\s\right)|^{2}.
\end{align*}
Note that
\begin{align*}
&\sum_{d\in \mathbb{D}}\left(\left(1-\frac{1}{2d}\right)^{|\supp(\s)|}-\left(1-\frac{1}{d}\right)^{|\supp(\s)|}\right)^2\\
\leq&\left(\sum_{d\in \mathbb{D}} \left(1-\frac{1}{2d}\right)^{|\supp(\s)|}-\left(1-\frac{1}{d}\right)^{|\supp(\s)|}\right)^2\leq1,
\end{align*}
and we complete the proof.
\end{proof}
By Lemma \ref{lehd} (iii), it follows that 
\begin{equation}\label{WHdT}
W_{\approx d}(H_d(T))\geq \frac{1}{16} W_{\approx d}(T).
\end{equation}
Before turning into the proof of the main theorem in this section, we provide the following lemma which is motivated by \cite[Lemma 29]{KKKMS2021}. 
\begin{lemma}\label{dgood}
Let $\mathbb{D}=\{2^k\}_{k\in \mathbb{Z}_+}$. For each $T\in \MP$ with $0\leq T\leq1$ and $1\leq p<2$, define 
\[
\mathbb{D}_{\geq}\coloneqq\Big\{d\in \mathbb{D}: \wyd(T)\geq\frac{\mathrm{var}\left(T\right)^2}{16\mathrm{Inf}^p(T)}\Big\}.
\]
Then we have
\[
\sum_{d\in \mathbb{D}_{\geq}} \wyd(T) \geq \frac{1}{2} \mathrm{var}\left(T\right).
\]
\end{lemma}

\begin{proof}
On the one hand, taking $d_0= \frac{4\mathrm{Inf}^p(T)}{\mathrm{var}\left(T\right)}$, we get that
\begin{equation}\label{dg1}
 \begin{aligned}
 	\sum_{d\geq d_0}\wyd(T) &=\sum_{|\supp(\s)|\geq d_0}|\widehat{T}(\s)|^2\\
 	&\leq \frac{1}{d_0}\sum_{|\supp(\s)|\geq d_0}|\supp(\s)||\widehat{T}(\s)|^2\\
&\leq\frac{1}{d_0} \sum_{\s\in\{0,1,2,3\}^n}|\supp(\s)||\widehat{T}(\s)|^2\\
&= \frac{\mathrm{Inf}(T)}{d_0}=\frac{\mathrm{Inf}(T)\mathrm{var}\left(T\right)}{4\mathrm{Inf}^p(T)}\leq\frac{1}{4}\mathrm{var}\left(T\right),
 \end{aligned} 
 \end{equation}
where the last inequality is due to Proposition \ref{dlili} (ii).

On the other hand, by the fact $|\{d\in \mathbb{D}:d\leq d_0\}|\leq \log\left(d_0\right)$ and the definition of $\mathbb{D}_{\geq}$, we get 
\begin{equation}\label{dg2}
\sum_{d< d_0,d\notin \mathbb{D}_{\geq}} \wyd(T)\leq \log\left(d_0\right)   \frac{\mathrm{var}\left(T\right)^2}{16\mathrm{Inf}^p(T)}  = \frac{\log(d_{0})}{4d_{0}}\var(T)\leq \frac{1}{4}\mathrm{var}\left(T\right).
\end{equation}
Combining \eqref{dg1} and \eqref{dg2}, we have
\begin{align*}
\sum_{d\in \mathbb{D}_{\geq}} \wyd(T)&=\sum_{d\in \mathbb{D}} W_{\approx d}(T)-\sum_{d\notin \mathbb{D}_{\geq}}W_{\approx d}(T)\\
&\geq \var(T)-\sum_{d\geq d_0}\wyd(T)-\sum_{d< d_0,d\notin \mathbb{D}_{\geq}} \wyd(T)\geq \frac{1}{2} \var(T).
\end{align*}
This completes the proof.
\end{proof}

Now we are at the position to give a proof of Theorem \ref{main 23} in detail. 

\begin{proof}[Proof of Theorem \ref{main 23}]
Take $T\in \MP$. We assume that 
\begin{equation}\label{KKLas}
\max_{j\in [n]}\|\dc_{j}(T)\|_{L_p}^p< e^{-1040K_p \frac{\mathrm{Inf}^p(T)}{\mathrm{var}\left(T\right)} },
\end{equation}
with $K_p=\frac{4}{(2-p)e}$. For $d\in \mathbb{D}_\geq$, we have
\[
\frac{\mathrm{Inf}^p(T)}{\mathrm{var}\left(T\right)}\wyd(T)\geq\frac{1}{4}\sqrt{\mathrm{Inf}^p(T)}\sqrt{\wyd(T)}.
\]
Using Corollary \ref{comlemma}, \eqref{WHdT}, \eqref{KKLas},  and Lemma \ref{lehd} (ii), we get that
\begin{equation}\label{equlastkkl}
\begin{aligned}
\mathrm{Inf}(H_d(T))+\var(H_d(T))&\geq \frac{1}{16} \log\left(\frac{1}{\max_{j\in [n]}\|\dc_j(H_{d}(T))\|_{L_p}^p}\right)\wyd(H_d(T))\\
&\quad -\frac{K_p}{16}\sqrt{\mathrm{Inf}^p(H_d(T))}\sqrt{\wyd(H_d(T))}\\
&\geq \frac{1}{256}\log\left(\frac{1}{\max_{j\in [n]}\|\dc_j(H_{d}(T))\|_{L_p}^p}\right)\wyd(T)\\
&\quad -\frac{K_p}{64}\sqrt{\mathrm{Inf}^p(H_d(T))}\sqrt{\wyd(T)}\\
&\geq  \frac{65K_{p}}{16}\frac{\mathrm{Inf}^p(T)}{\mathrm{var}\left(T\right)}\wyd(T)-\frac{K_{p}}{16}\sqrt{\mathrm{Inf}^p(T)}\sqrt{\wyd(T)}\\
&\geq 4K_p \frac{\mathrm{Inf}^p(T)}{\mathrm{var}\left(T\right)}\wyd(T).
\end{aligned}
\end{equation}
Combining Lemma \ref{lehd} (iv), \eqref{equlastkkl} with  Lemma \ref{dgood} yields that
\begin{align*}
2\mathrm{Inf}^p(T)&\geq 2\mathrm{Inf}(T)\geq \mathrm{Inf}(T)+ \var(T)\geq \sum_{d\in \mathbb{D}_{\geq}} \mathrm{Inf}(H_d(T))+\sum_{d\in \mathbb{D}_{\geq}}\var(H_d(T))\\
&\geq 4K_p \frac{\mathrm{Inf}^p(T)}{\mathrm{var}\left(T\right)}\sum_{d\in \mathbb{D}_{\geq}}\wyd(T)\geq  2K_p\frac{\mathrm{Inf}^p(T)}{\mathrm{var}\left(T\right)} \mathrm{var}\left(T\right)\geq\frac{8}{e}\mathrm{Inf}^p(T),
\end{align*}
which is a contradiction.
\end{proof}

The following counterexample demonstrates the failure of Theorem \ref{main 23} for the case $p=2$ even in the commutative hypercube setting, which means that, for bounded elements, the \emph{dimension free} KKL inequality invoking $L_{p}$-influences with $1\leq p<2$ may be the best possible.

\begin{remark}\label{rem p2}
Define $f:\{0,1\}^n\to [0,1]$ as follows
\[
f\coloneqq\frac{1}{2}+\frac{1}{2n}\sum_{j=1}^{n}r_{j},
\]
where $\{r_{j}\}_{j=1}^{n}$ is an i.i.d. Rademacher sequence. Then, we have
\[
\|\dc_j (f)\|_{L_2}^2=\frac{1}{4n^{2}},\quad \forall j\in [n],
\]
and
\[
\var(f)=\mathrm{Inf}(f)=\frac{1}{4n},
\]
which disproves Theorem \ref{main 23} for the case $p=2$ in the hypercube setting.
\end{remark}

We conclude this section with some comments on the quantum Talagrand influence inequality. Firstly, combining the idea in \cite{KKKMS2021} with the random restriction treatment presented in Section \ref{quantum KKL}, we can derive the following quantum Talagrand influence inequality.

\begin{theorem}[Quantum Talagrand influence inequality]\label{Talagrand}
For $1\leq p<2$, there exists a constant $C_{p}>0$ depending only on $p$, such that for each $T\in \M_{2^n}$ with $0\leq T\leq1$, we have
\begin{equation}\label{quantum Talagrand}
\var(T)\leq C_{p}\sum_{j=1}^{n}\frac{\|\dc_{j}(T)\|^{p}_{L_{p}}}{\log\left(1/\|\dc_{j}(T)\|^{p}_{L_{p}}\right)}
\end{equation}
where $C_p=O\left(\frac{1}{(2-p)^2}\right)$ as $p\to 2$.
\end{theorem}
\noindent
After completing our paper, we learned that Blecher et al. \cite{BGX2024} have obtained a similar quantum Talagrand influence inequality with some interesting higher-order extensions. Therefore, we have chosen not to include the detailed proof of Theorem \ref{Talagrand} here and instead refer interested readers to \cite{BGX2024} for the proof. Secondly, the constant $C_{p}=O\left(\frac{1}{(2-p)^{2}}\right)$ appeared in the quantum Talagrand influence inequality is not the best possible. However, at this time of writing, we can not apply the random restriction method to achieve the constant $C_{p}=O\left(\frac{1}{2-p}\right)$ as stated in Theorem \ref{RWZ-Talagrand}. Finally, we shall mention here that \eqref{quantum Talagrand} is a straightforward consequence of Theorem \ref{RWZ-Talagrand} by using Proposition \ref{dlili} (i).

%

\section{Quantum Eldan-Gross inequality}\label{quantum Eldan-Gross}
In this section, we aim to derive a quantum Eldan-Gross inequality via the random restriction technique, and apply it to obtain several quantum KKL-type inequalities. Our approach of the quantum Eldan-Gross inequality is inspired by the techniques developed by Keller and Kinder \cite{KK2013} and Eldan et al. \cite{EKLM2022}.
\subsection{Estimate on Fourier spectrum}
This subsection aims to establish the following noncommutative analogue of \cite[Lemma 5]{KK2013}. For $T\in \M_{2^{n}}$ and $J\subseteq [n]$, we denote  $M_J(T)\coloneqq \sum_{j\in J}\|\dc_{j}(T)\|^{2}_{L_{1}}$ and $M(T)\coloneqq \sum_{j=1}^{n}\|\dc_{j}(T)\|^{2}_{L_{1}}$.

\begin{theorem}\label{KK18_5}
For $d\geq 1$ and projection $T\in \M_{2^{n}}$, if $M(T)\leq e^{-2d}$, then we have
\[
\sum_{|\supp(\s)|=d}\widehat{T}(\s)^{2}\leq \frac{6e}{d}\left(\frac{2e}{d}\right)^{d}M(T)\left(\log\left(\frac{d}{M(T)}\right)\right)^{d}.
\]
\end{theorem}

The key ingredient in the proof of Theorem \ref{KK18_5} is the following technical result, which is a noncommutative analogue of the estimate in \cite[eq. (12)]{KK2013}.

\begin{proposition}\label{key lemma 1}
Let $T\in \M_{2^{n}}$ be a projection with $M(T)\leq e^{-2d}$. Then, for each $J\subseteq [n]$, we have
\begin{equation*}
\begin{aligned}
\sum_{j\in J}\sum_{\substack{\supp\left(\s\right)\subseteq J^{c},\\ |\supp\left(\s\right)|=d-1,\alpha \in\{1,2,3\}}}&|\widehat{T}\left(\s\oplus e_j^{\alpha}\right)|^2
\leq 6\left(\frac{2e}{d}\right)^{d}M_J(T)\left(\log\left(\frac{1}{M_J(T)}\right)\right)^{d}.
\end{aligned}
\end{equation*}
\end{proposition}
Since the proof of Proposition \ref{key lemma 1} is a little bit lengthy, we will postpone its proof, and  turn to demonstrate how it can be used to prove Theorem \ref{KK18_5} at first. Recall that we identify random subset $J$ with the vector in $(\{0,1\}^{n},\mu_{\frac{1}{d}})$, where
\[
\mu_{\frac{1}{d}}(\{x\})=\left(\frac{1}{d}\right)^{\sum_{j=1}^{n}x_{j}}\left(1-\frac{1}{d}\right)^{n-\sum_{j=1}^{n}x_{j}},\quad x=(x_{j})_{j=1}^{n}\in \{0,1\}^{n}.
\]
\begin{proof}[Proof of Theorem \ref{KK18_5}]
Firstly, note here that
for each $\s\in \{0,1,2,3\}^{n}$ with $|\supp(\s)|=d$, we know the probability of $J$ such that $\s=\mathbf{u}\oplus e^{\alpha}_{j}$ with $j\in J$ and $\supp(\mathbf{u})\subseteq J^{c}$ is $\binom{d}{1}\cdot\frac{1}{d}(1-\frac{1}{d})^{d-1}=(1-\frac{1}{d})^{d-1}\geq \frac{1}{e}$. Hence, 
\begin{align}\label{kk18left}
&\E_{J}\left(\sum_{\substack{\supp\left(\s\right)\subseteq J^{c}, |\supp\left(\s\right)|=d-1\\j\in J,\alpha \in\{1,2,3\}}}\widehat{T}\left(\s\oplus e_j^{\alpha}\right)^2\right)\\\nonumber
=&\left(1-\frac{1}{d}\right)^{d-1}\sum_{\substack{\s\in\{0,1,2,3\}^{n},\\|\supp(\s)|=d}}\widehat{T}(\s)^{2}\geq\left(\frac{1}{e}\right)\sum_{\substack{\s\in\{0,1,2,3\}^{n},\\|\supp(\s)|=d}}\widehat{T}(\s)^{2}.
\end{align}

Secondly, assume Proposition \ref{key lemma 1} holds and note that $x\mapsto x\log(\frac{1}{x})^{d}$ is a concave function on $(0,e^{d})$. Then it follows from the assumption $M(T)\leq e^{-2d}$ and the Jensen inequality that
\begin{equation}\label{kk18right}
\begin{aligned}
&\mathbb{E}_{J}\left[6\left(\frac{2e}{d}\right)^{d}\left(\sum_{j\in J}\|\dc_{j}(T)\|^{2}_{L_{1}}\right)\left(\log\left(\frac{1}{\sum_{j\in J}\|\dc_{j}(T)\|^{2}_{L_{1}}}\right)\right)^{d}\right]\\
\leq&6\left(\frac{2e}{d}\right)^{d}\left(\mathbb{E}_{J}\left(\sum_{j\in J}\|\dc_{j}(T)\|^{2}_{L_{1}}\right)\right)\left(\log\left(\frac{1}{\mathbb{E}_{J}(\sum_{j\in J}\|\dc_{j}(T)\|^{2}_{L_{1}})}\right)\right)^{d}\\
=&6\left(\frac{2e}{d}\right)^{d}\left(\frac{M\left(T\right)}{d}\right)\left(\log\left(\frac{d}{M\left(T\right)}\right)\right)^{d},
\end{aligned}
\end{equation}
where we used $\mathbb{E}_{J}(\mathds{1}_{J}(j))=\frac{1}{d}$ in the last equality. Hence, combing \eqref{kk18left}, \eqref{kk18right} with Proposition \ref{key lemma 1} yields the desired inequality.
\end{proof}
We conclude this subsection with the following noncommutative analogue of \cite[Theorem 3.4]{EKLM2022}.
\begin{theorem}\label{main spectral}
Suppose that $T\in \M_{2^{n}}$ is a projection. Then  
\begin{equation}\label{main spectral estimate 1}
\sum_{1\leq |\supp(\s)|\leq \frac{1}{10}\log(1/M(T))}\widehat{T}(\s)^{2}\leq 12e M(T)^{2/5}.
\end{equation}
\end{theorem}
\begin{proof}
For the given projection $T$, recall here that
\[
W_{=l}(T)=\sum_{|\supp(\s)|=l}|\widehat{T}(\s)|^{2}.
\]
For $1\leq d\leq\frac{1}{10}\log(1/M(T))$, we have $M(T)\leq e^{-10d}$, and then it follows from Theorem \ref{KK18_5} that
\begin{equation}\label{main estimate 9}
\begin{split}
&W_{1\leq d\leq \frac{1}{10}\log(1/M(T))}(T)\\
=&\sum_{l=1}^{\left\lceil \frac{1}{10}\log(1/M(T))\right\rceil}W_{=l}(T)\\
\leq& 6eM(T)\sum_{l=1}^{\left\lceil \frac{1}{10}\log(1/M(T))\right\rceil}\frac{1}{l}\left(\frac{2e}{l}\right)^{l}\left(\log\left(\frac{l}{M(T)}\right)\right)^{l}\\
=&6eM(T)\sum_{l=1}^{\left\lceil \frac{1}{10}\log(1/M(T))\right\rceil}\frac{1}{l}\left(\frac{2e}{l}\right)^{l}\left(\log(l)+\log(1/M(T))\right)^{l}\\
\leq&6eM(T)\left(\frac{4e\log(1/M(T))}{l}\right)^{l}\left(\sum_{l=1}^{\left\lceil \frac{1}{10}\log(1/M(T))\right\rceil}\frac{1}{l}\right),
\end{split}
\end{equation}
where we used $\log(l)\leq \log(1/M(T))$ for $1\leq l\leq\frac{1}{10}\log(1/M(T))$ in the last inequality.

Note that the function $l\mapsto \left(\frac{4e\log(1/M(T))}{l}\right)^{l}$ is increasing for $l\in[1,4\log(1/M(T))]$. Hence, we estimate \eqref{main estimate 9} as follows
\begin{equation}\label{main estimate 10}
\begin{split}
W_{1\leq d\leq \frac{1}{10}\log(1/M(T))}(T)&\leq 6e M(T)(40e)^{\frac{1}{10}\log(1/M(T))}\sum_{l=1}^{\left\lceil \frac{1}{10}\log(1/M(T))\right\rceil}\frac{1}{l}\\
&\leq 6eM(T)^{1/2}\sum_{l=1}^{\left\lceil \frac{1}{10}\log(1/M(T))\right\rceil}\frac{1}{l}\\
&\leq 6e M(T)^{1/2}\left(\log\log\left(\frac{1}{M(T)^{1/10}}\right)+1\right)\\
&\leq 12e M(T)^{2/5},
\end{split}
\end{equation}
where we used $40e\leq e^{5}$ and $\sum_{l=1}^{ \left\lceil \frac{1}{10}\log(1/M(T))\right\rceil}\frac{1}{l}\leq \log\log\left(\frac{1}{M(T)^{1/10}}\right)+1$ in the second and third inequalities, respectively.
\end{proof}

\subsection{The quantum Talagrand-type isoperimetric inequality}
In this subsection, we apply the semigroup technique to derive Theorem \ref{quantum Ta}. The proof of the theorem relies on the following inequalities, which are of independent interest. The first key ingredient in our proof is the following quantum Buser-type inequality.
\begin{theorem}[Quantum Buser-type inequality]\label{quantum Buser inequality}
For each $t\geq 0$ and $1\leq p\leq 2$, the following inequality holds for every $T\in\MP$
\[
\left\|T-e^{-t\La}(T)\right\|_{L_{p}}\leq \sqrt{2t}\left\|\left(\sum_{j=1}^{n}|\dc_{j}(T)^{2}|\right)^{1/2}\right\|_{L_{p}}.
\]
\end{theorem}

To establish Theorem \ref{quantum Buser inequality}, we need the following sequence of lemmas, which are noncommutative analogies to their commutative correspondences.

\begin{theorem}[Local Reverse Poincar\'e inequality]\label{reverse local Poincare inequality}
For each $T\in \MP$, we have
\begin{equation*}
(e^{2t}-1)\sum\limits_{j=1}^{n}|e^{-t\La}\dc_{j}(T)|^{2}\leq e^{-t\La}(|T|^{2})-|e^{-t\La}(T)|^{2}.
\end{equation*}
\end{theorem}

\begin{proof}
On the one hand, applying Proposition \ref{curvature condition} (iii), we get
\begin{equation}\label{curvature condition 1}
e^{-s\La}\left[\sum\limits_{j=1}^{n}\left|\dc_{j}e^{-(t-s)\La}(T)\right|^{2}\right]\geq e^{2s}\sum\limits_{j=1}^{n}|\dc_{j}e^{-t\La}(T)|^{2},
\end{equation}
for every $0\leq s\leq t$.

On the other hand,
\begin{equation}\label{curvature condition 2}
\begin{split}
  &e^{-t\La}(T^{*}T)-(e^{-t\La}(T))^{*}(e^{-t\La}(T))\\
=&\int_{0}^{t}\frac{\partial}{\partial s}e^{-s\La}\left[\left(e^{-(t-s)\La}(T)\right)^{*}\left(e^{-(t-s)\La}(T)\right)\right]~ds\\
=&\int_{0}^{t}-\La e^{-s\La}\left(\left|e^{-(t-s)\La}(T)\right|^{2}\right)+e^{-s\La}\left(\left(e^{-(t-s)\La}(T)\right)^{*}\La e^{-(t-s)\La}(T)\right)\\
&\quad+ e^{-s\La}\left(\left(\La e^{-(t-s)\La}(T)\right)^{*}e^{-(t-s)\La}(T)\right)~ds\\
=&\int_{0}^{t}e^{-s\La}\Big{[}-\La\left( |e^{-(t-s)\La}(T)|^{2}\right)+e^{-(t-s)\La}(T)\left(\La e^{-(t-s)\La}(T)\right)^{*}\\
&\quad +\left(\La e^{-(t-s)\La}(T)\right)^{*} e^{-(t-s)\La}(T)\Big{]}~ds\\
=&2\int_{0}^{t}e^{-s\La}\left(\sum\limits_{j=1}^{n}\left|\dc_{j} e^{-(t-s)\La}(T)\right|^{2}\right)~ds,
\end{split}
\end{equation}
where the last line follows from Proposition \ref{curvature condition} (i).

Combining \eqref{curvature condition 1} and \eqref{curvature condition 2} yields that
\begin{equation*}
\begin{split}
e^{-t\La}(|T|^{2})-|e^{-t\La}(T)|^{2}=&2\int_{0}^{t}e^{-s\La}\left(\sum\limits_{j=1}^{n}\left|\dc_{j}e^{-(t-s)\La}(T)\right|^{2}\right)~ds\\
\geq&2\int_{0}^{t}e^{2s}\sum\limits_{j=1}^{n}\left|\dc_{j}e^{-t\La}(T)\right|^{2}~ds\\
=&\sum\limits_{j=1}^{n}\left|\dc_{j}e^{-t\La}(T)\right|^{2}\int_{0}^{2t}e^{s}~ds\\
=&(e^{2t}-1)\sum\limits_{j=1}^{n}\left|\dc_{j}e^{-t\La}(T)\right|^{2},
\end{split}
\end{equation*}
which is the desired inequality.
\end{proof}

Thanks to the local reverse Poincar\'e inequality (i.e., Theorem \ref{reverse local Poincare inequality}), we obtain the following \emph{gradient estimate} that is dual to the quantum Byser-type inequality.

\begin{lemma}\label{dual Buser inequality}
For each $T\in \MP$ and $2\leq p\leq \infty$, we have
\[
\left\|\left(\sum\limits_{j=1}^{n}\left|\dc_{j}e^{-t\La}(T)\right|^{2}\right)^{1/2}\right\|_{L_{p}}\leq \frac{1}{\sqrt{2t}}\|T\|_{L_{p}},~t\geq 0.
\]
\end{lemma}

\begin{proof}
Since $e^{2t}-1\geq 2t$, it follows from the reverse local Poincar\'e inequality (i.e., Theorem \ref{reverse local Poincare inequality}) that for each $t\geq 0$ we have
\begin{equation*}\label{reverse local Poincare inequality 1}
2t\sum\limits_{j=1}^{n}|\dc_{j}e^{-t\La}(T)|^2\leq e^{-t\La}(|T|^2)-|e^{-t\La}(T)|^2\leq e^{-t\La}(|T|^2),
\end{equation*}
which further implies
\[
\left(\sum\limits_{j=1}^{n}\left|\dc_{j}e^{-t\La}(T)\right|^{2}\right)^{1/2}\leq \frac{1}{\sqrt{2t}}\left(e^{-t\La}(T^{*}T)\right)^{1/2}.
\]
Therefore, we have 
\begin{align*}
\left\|\left(\sum\limits_{j=1}^{n}\left|\dc_{j}e^{-t\La}(T)\right|^{2}\right)^{1/2}\right\|_{L_{p}}
&\leq\frac{1}{\sqrt{2t}}\left\|e^{-t\La}(|T|^{2})\right\|^{1/2}_{L_{p/2}}\leq \frac{1}{\sqrt{2t}}\|T\|_{L_{p}},
\end{align*}
where the last inequality is due to the fact that $p/2\geq 1$ and
\begin{equation*}\label{dual Buser 2}
\left\|e^{-t\La}(|T|^{2})\right\|^{1/2}_{L_{p/2}}\leq \left\||T|^{2}\right\|^{1/2}_{L_{p/2}}=\|T\|_{L_{p}}.
\end{equation*}
\end{proof}
%
We now prove the quantum Buser-type inequality via duality.
\begin{proof}[Proof of Theorem \ref{quantum Buser inequality}]
Since, for each $t\geq 0$,
\[
T-e^{-t\La}(T)=-\int_{0}^{t}\frac{\partial}{\partial s}e^{-s\La}(T)~ds=\sum\limits_{j=1}^{n}\int_{0}^{t}\dc_{j} e^{-s\La}(T)~ds.
\]
Hence, for every $1\leq p\leq 2$, there exists $u\in L_{p^{\prime}}$ with $\|u\|_{L_{p^{\prime}}}=1$, $\frac{1}{p}+\frac{1}{p^{\prime}}=1$, and
\begin{equation*}\label{Buser-type inequality 1}
\begin{split}
\left\|T-e^{-t\La}(T)\right\|_{L_{p}}&=\tr\left(u\cdot\left(\sum\limits_{j=1}^{n}\int_{0}^{t}\dc_{j} e^{-s\La}(T)~ds\right)\right)\\
&=\int_{0}^{t}\sum\limits_{j=1}^{n}\tr\left(u\cdot\left(\dc_{j}e^{-s\La}(T)\right)\right)~ds\\
&=\int_{0}^{t}\sum\limits_{j=1}^{n}\tr\left(\left(\dc_{j}e^{-s\La}(u)\right)\cdot(\dc_{j}(T))\right)~ds\\
&\leq \int_{0}^{t}\left\|\left(\sum_{j=1}^{n}|\dc_{j}e^{-s\La}(u^{*})|^{2}\right)^{1/2}\right\|_{L_{p^{\prime}}}\cdot\left\|\left(\sum_{j=1}^{n}|\dc_{j}(T)^{2}|\right)^{1/2}\right\|_{L_{p}}ds.
\end{split}
\end{equation*}
Thanks to Lemma \ref{dual Buser inequality}, we obtain that
\begin{equation*}\label{Buser-type inequality 3}
\begin{split}
\left\|T-e^{-t\La}(T)\right\|_{L_{p}}&\leq \int_{0}^{t}\frac{1}{\sqrt{2s}}ds~\|u\|_{L_{p^{\prime}}}\left\|\left(\sum_{j=1}^{n}|\dc_{j}(T)^{2}|\right)^{1/2}\right\|_{L_{p}}\\
&\leq \sqrt{2t}\left\|\left(\sum_{j=1}^{n}|\dc_{j}(T)^{2}|\right)^{1/2}\right\|_{L_{p}}.
\end{split}
\end{equation*}
\end{proof}

Before turning to the Talagrand-type inequality, we show the following two elementary lemmas.

\begin{lemma}\label{fundamental identity}
For each  projection $T\in \MP$ and $t\geq 0$, we have
\[
\var(T)\leq \left\|T-e^{-t\La}(T)\right\|_{L_{1}}+\var\left(e^{-t\La/2}(T)\right).
\]
\end{lemma}

\begin{proof}
Let $T\in \MP$ be a projection. It follows from the trace preserving of $\{e^{-t\La}\}_{t\geq 0}$ that
\[
\var\left(e^{-t\La/2}(T)\right)=\tr(T\cdot e^{-t\La}(T))-\tr(T)^{2},~\forall~t\geq 0.
\] 
Note that $\var(T)=\tr(T)-\tr(T)^{2}$. Therefore, we have
\begin{align*}\label{elementary estimate 1}
\var(T)-\var\left(e^{-t\La/2}(T)\right)&=\tr(T-T\cdot e^{-t\La}(T))\\
&=\tr\left(T(T- e^{-t\La}(T))\right)\\
&\leq\left\|T(T-e^{-t\La}(T))\right\|_{L_{1}}\\
&\leq \|T-e^{-t\La}(T)\|_{L_{1}}.
\end{align*}
The desired assertion follows.
\end{proof}

\begin{lemma}\label{high degree}
Let $T\in \MP$ be a projection. For each $d\in \mathbb{N}$, we have
\begin{equation*}\label{key estimate 9}
\left\|\left(\sum_{j=1}^{n}|\dc_{j}(T)|^{2}\right)^{1/2}\right\|_{L_{1}}\geq \frac{1}{4}\sqrt{d}\,W_{\geq d}(T).
\end{equation*}
\end{lemma}

\begin{proof}
Fix  $d\in \mathbb{N}$. Applying Lemma \ref{fundamental identity} (with $t=\frac{1}{d}$), we get
\begin{align*}
\left\|T-e^{-\frac{\La}{d}}(T)\right\|_{L_{1}}&\geq \var(T)-\var\left(e^{-\frac{\La}{2d}}(T)\right)\\
&=\sum_{|\supp(\s)|\geq 1}(1-e^{-\frac{|\supp(\s)|}{d}})\widehat{T}(\s)^{2}\\
&\geq\sum_{|\supp(\s)|\geq d}(1-e^{-\frac{|\supp(s)|}{d}})\widehat{T}(\s)^{2}\\
&\geq(1-e^{-1})W_{\geq d}(T)\geq \frac{1}{2}W_{\geq d}(T).
\end{align*}
By Theorem \ref{quantum Buser inequality}, we have  
\[
\left\|\left(\sum_{j=1}^{n}|\dc_{j}(T)|^{2}\right)^{1/2}\right\|_{L_{1}}\geq \sqrt{\frac{d}{2}}\left\|T-e^{-\frac{\La}{d}}(T)\right\|_{L_{1}} \geq \frac{\sqrt{d}}{4}W_{\geq d}(T).
\]
\end{proof}

Before providing the proof of the quantum Talagrand-type isoperimetric inequality, i.e., Theorem \ref{quantum Ta}, we derive the following moment comparison lemma via hypercontractivity of the semigroup $\{e^{-t\mathcal{L}}\}_{t\geq 0}$ (see \cite[Theorem 46]{MO2010}).
\begin{lemma}[moment comparison]\label{lemmaHC}
Let $T\in \M_{2^{n}}$ with degree at most $k$. For every $r\geq 2$, we have
\[
\|T\|_{L_{r}}\leq\left(r-1\right)^{k/2}\|T\|_{L_{2}}.
\]
\end{lemma}
\begin{proof}
Note that $T$ is of degree at most $k$, that is, $T=\sum_{|\mathrm{supp}(\mathbf{s})|\leq k}\widehat{T}(\mathbf{s})\sigma_{\mathbf{s}}$. For each $r\geq 2$, take $t_{0}=\frac{\log(r-1)}{2}$. Applying the hypercontractivity of $\{e^{-t\mathcal{L}}\}_{t\geq 0}$ (see \eqref{MO hypercontractivity} above), we have
\begin{equation*}\label{hyper-estimate1}
\left\|e^{-\frac{\log(r-1)}{2}\mathcal{L}}(T)\right\|_{L_{r}}\leq \|T\|_{L_{2}},
\end{equation*}
which implies
\begin{equation*}\label{hyper-estimate2}
\|T\|_{L_{r}}=\left\|e^{-t\frac{\log(r-1)}{2}\mathcal{L}}\circ e^{t\frac{\log(r-1)}{2}\mathcal{L}}(T)\right\|_{L_{r}}\leq \left\|e^{t\frac{\log(r-1)}{2}\mathcal{L}}(T)\right\|_{L_{2}}.
\end{equation*}
Now the desired inequality follows from the fact
\begin{equation*}\label{hyper-estimate3}
\left\|e^{\frac{\log(r-1)}{2}\mathcal{L}}(T)\right\|^{2}_{L_{2}}=\sum_{|\mathrm{supp}(\mathbf{s})|\leq k}e^{\log(r-1)|\mathrm{supp}(\mathbf{s})|}|\widehat{T}(\mathbf{s})|^{2}\leq(r-1)^{k}\|T\|^{2}_{L_{2}}.
\end{equation*}
\end{proof}

\begin{proof}[Proof of Theorem \ref{quantum Ta}]
We first consider the case $\var(T)\geq e^{-16}$. Note that $\var(T)=W_{\geq1}(T)$ and
\[
\sqrt{\log\left(\frac{1}{\var(T)}\right)}\leq 4.
\]
It follows from Lemma \ref{high degree} (with $d=1$ there) that
\begin{align*}
\var(T)	\sqrt{\log\left(\frac{1}{\var(T)}\right)}\leq 4 W_{\geq 1}(T)\leq 16 \left\|\left(\sum_{j=1}^{n}|\dc_{j}(T)|^{2}\right)^{1/2}\right\|_{L_{1}}.
\end{align*}

Now we deal with the case $\var(T)\leq e^{-16}$. Let
\[
d\coloneqq \left\lceil\frac{1}{16}\log\left(\frac{1}{\var(T)}\right)\right\rceil.
\]
For such $d$, we claim that 
\begin{equation}\label{claim Wd}
	W_{\geq d}(T)\geq \frac{1}{2}\var(T).
\end{equation}
Once the claim is proved, applying Lemma \ref{high degree}, we have
\begin{equation*}
\left\|\left(\sum_{j=1}^{n}|\dc_{j}(T)|^{2}\right)^{1/2}\right\|_{L_{1}}\geq \frac{1}{4}\sqrt{\frac{1}{16}\log\left(\frac{1}{\var(T)}\right)}W_{\geq d}(T)\geq \frac{1}{32} \var(T)	\sqrt{\log\left(\frac{1}{\var(T)}\right)}.
\end{equation*}
This is the desired assertion.

It remains to verify the claim.
Since $T\in \M_{2^{n}}$ is a projection, it follows that $\var(T)=\tr(T)(1-\tr(T))\leq e^{-16}$. We assume without loss of generality that $\tr(T)<\frac{1}{2}$; otherwise, it suffices to replace $T$ by $\mathbf{1}-T$. Hence,
\begin{equation}\label{key identity9}
2\var(T)=2\tr(T)(1-\tr(T))\geq \tr(T).
\end{equation}
Write
\[
\mathrm{Rad}_{\leq d}(T)=\sum_{|\supp(\s)|\leq d} \widehat{T}(\s)\sigma_{\s}.
\]
By the H\"older inequality, Lemma \ref{lemmaHC} and \eqref{key identity9}, we deduce that
\begin{align*}
W_{\leq d}(T)&=\tr\left(T\cdot \mathrm{Rad}_{\leq d}(T)\right)\leq \|T\|_{4/3}\cdot\|\mathrm{Rad}_{\leq d}(T)\|_{L_{4}}\\
&\leq 3^{d/2}\|\mathrm{Rad}_{\leq d}(T)\|_{L_{2}}\|T\|_{L_{4/3}}=3^{d/2}\|\mathrm{Rad}_{\leq d}(T)\|_{L_{2}}\tr(T)^{3/4}\\
&\leq 3^{d/2}\cdot 2^{3/4}\|\mathrm{Rad}_{\leq d}(T)\|_{L_{2}}\var(T)^{3/4}=3^{d/2}\cdot 2^{3/4}W_{\leq d}(T)^{1/2}\var(T)^{3/4}.
\end{align*}
Note that $3^{d/2}\leq e^{d}\leq \var(T)^{-1/16}$ and $\var(T)\leq e^{-16}$. We have
\[
W_{\leq d}(T)^{1/2}\leq \frac{2^{3/4}\var(T)^{3/4}}{\var(T)^{1/16}}\leq\var(T)^{1/2}2^{3/4}e^{-3}\leq\frac{1}{2}\var(T)^{1/2},
\]
which further implies
\begin{equation*}
W_{\geq d}(T)\geq \var(T)-W_{\leq d}(T)\geq \frac{3}{4}\var(T)\geq \frac{1}{2}\var(T).
\end{equation*}
We have verified the claim \eqref{claim Wd}, and the proof is complete.
\end{proof}

\subsection{The quantum Eldan-Gross inequality and KKL-type inequalities}
In this subsection, we provide a details proof of the quantum Eldan-Gross inequality and apply it to deduce two quantum KKL-type inequalities and a stability result for the quantum KKL theorem with respect to the $L_{1}$-influences.
\begin{proof}[Proof of Theorem \ref{quantum E-G}]
Let $T$ be a projection in $\M_{2^{n}}$ and denote $M(T)\coloneqq \sum_{j=1}^{n}\|\dc_{j}(T)\|^{2}_{L_{1}}$ for simplicity. If $M(T)\geq \var(T)^{15}$, it follows that
\begin{equation}\label{key estimate 14}
\begin{split}
\var(T)\sqrt{\log\left(1+\frac{1}{M(T)}\right)}&\leq \var(T)\sqrt{\log\left(1+\frac{1}{\var(T)^{10}}\right)}\\
&\leq 4\var(T)\sqrt{\log\left(\frac{1}{\var(T)}\right)},
\end{split}
\end{equation}
where we used $\var(T)\leq \frac{1}{4}$ (i.e., $\max_{\alpha\in[0,1]}\alpha-\alpha^{2}\leq \frac{1}{4}$) in the second inequality. Apply Theorem \ref{quantum Ta} to \eqref{key estimate 14} yields the desired quantum Eldan-Gross inequality for the case $M(T)\geq \var(T)^{15}$.

If $M(T)\leq \var(T)^{15}$, we let $d=\frac{1}{10}\log(1/M(T))$ and apply Theorem \ref{main spectral} to get that
\begin{equation}\label{key estimate 15}
\sum_{1\leq |\supp(\s)|\leq \frac{1}{10}\log(1/M(T))}\widehat{T}(\s)^{2}\leq 12e\cdot M(T)^{2/5}\leq 12e\cdot\var(T)^{6}\leq \left(\frac{12e}{4^{5}}\right)\var(T),
\end{equation}
where we used $\var(T)\leq \frac{1}{4}$ in the last inequality. By \eqref{key estimate 15}, it follows that
\begin{equation}\label{key estimate 16}
W_{>d}(T)=\var(T)-\sum_{1\leq |\supp(\s)|\leq d}\widehat{T}(\s)^{2}\geq \left(1-\frac{12e}{4^{5}}\right)\var(T)\geq \frac{1}{2}\var(T).
\end{equation}
Combining Lemma \ref{high degree} and \eqref{key estimate 16}, we get that
\[
\var(T)\sqrt{\log\left(1+\frac{1}{\log(M(T))}\right)}\leq K\left\|\left(\sum_{j=1}^{n}|\dc_{j}(T)|^{2}\right)^{1/2}\right\|_{L_{1}},
\]
where we used $1+\log(1/M(T))\leq 2\log(1/M(T))$ when $M(T)\leq \var(T)^{15}\leq 1/2^{30}$.
\end{proof}

The following definition is motivated by Definition \ref{def bqb} and Remark \ref{equivalence}.
\begin{definition}
A projection $T\in \mathbb{M}_{2^n}$ is called balanced if $\mathrm{var}(T)=\frac{1}{4}$.
\end{definition}
As the first application of our noncommutative Eldan-Gross inequality, we derive the following KKL-type inequality in the CAR algebra, which is essentially due to Rouz\'e, Wirth and Zhang \cite{RWZ2024}.

\begin{theorem}[Rouz\'e-Wirth-Zhang]\label{quantum KKL I}
For each balanced projection $T\in \mathbb{M}_{2^n}$, there exists a universal constant $C>0$, such that
\[
\max_{j\in [n]}\|\dc_{j}(T)\|_{L_{1}}\geq \frac{C\sqrt{\log(n)}}{n}.
\]
\end{theorem}

\begin{proof}
For each balanced projection $T\in  \mathbb{M}_{2^n}$, we apply Theorem \ref{quantum E-G} to get
\begin{equation}\label{quantum KKL 1}
\begin{split}
\frac{1}{4}\sqrt{\log(1+\frac{1}{\sum_{j=1}^{n}\|\dc_{j}(T)\|^{2}_{L_{1}}})}&\leq K\left\|\sum_{j=1}^{n}\left|\dc_{j}(T)\right|^{2}\right\|^{1/2}_{L_{1/2}}\leq K\sum_{j=1}^{n}\|\dc_{j}(T)\|_{L_{1}},
\end{split}
\end{equation}
where $K>0$ is a universal constant. It follows from \eqref{quantum KKL 1} that
\begin{equation}\label{prior estimate 1}
\frac{1}{16}\leq \frac{(KnB)^{2}}{\log\left(1+\frac{1}{nB^{2}}\right)},
\end{equation}
where  $B=\max_{j\in[n]}\|\dc_{j}(T)\|_{L_{1}}$.

Choose a positive constant $C$ such that $16KC<1$. If $M\geq \frac{C\sqrt{\log(n)}}{n}$, there is nothing to prove. We assume $B<\frac{C\sqrt{\log(n)}}{n}$ from now on. Choose a sufficient large $n$ such that
\[
\frac{n}{C^{2}\log(n)}-n^{1/2}+1>0.
\]
It follows from $B<\frac{C\sqrt{\log(n)}}{n}$ that
\begin{equation}\label{prior estimate 2}
\frac{(KnB)^{2}}{\log\left(1+\frac{1}{nB^{2}}\right)}<\frac{K^{2}C^{2}\log(n)}{\log\left(1+\frac{n}{C^{2}\log(n)}\right)}.
\end{equation}
Note that $4\sqrt{2}KC<1$ and $\frac{n}{C^{2}\log(n)}-n^{1/2}+1>0$ imply that
\begin{equation}\label{prior estimate 3}
(4KC)^{2}\log(n)<\frac{1}{2}\log(n)<\log\left(1+\frac{n}{C^{2}\log(n)}\right).
\end{equation}
Substituting \eqref{prior estimate 3} to \eqref{prior estimate 2} entails that
\[
\frac{(KnB)^{2}}{\log\left(1+\frac{1}{nB^{2}}\right)}<\frac{K^{2}C^{2}\log(n)}{\log\left(1+\frac{n}{C^{2}\log(n)}\right)}<\frac{1}{16},
\]
which contradicts to \eqref{prior estimate 1}.
\end{proof}

As the second application of the noncommutative Eldan-Gross inequality, we derive the following quantum KKL-type inequality, which is of independent interest.
\begin{theorem}\label{quantum KKL II}
There exists a universal constant $C>0$ such that for every $\varepsilon\in(0,1)$ and balanced projection $T\in \MP$, one of the following inequalities holds:
\begin{enumerate}[\rm(i)]
\item $\max_{j\in[n]}\|\dc_{j}(T)\|^{2}_{L_{2}}\geq \frac{C\varepsilon\log(n)}{n}$;
\item $\max_{j\in[n]}\|\dc_{j}(T)\|_{L_{1}}\geq \frac{C}{n^{(1+\varepsilon)/2}}$.
\end{enumerate}
\end{theorem}

\begin{proof}
For each $\varepsilon>0$ and balanced projection $T\in \MP$, by the quantum Eldan-Gross inequality (i.e., Theorem \ref{quantum E-G}), there exists a universal constant $K>0$ such that
\begin{equation}\label{E-G estimate}
\begin{split}
\frac{1}{4}\sqrt{\log\left(1+\frac{1}{\sum_{j=1}^{n}\|\dc_{j}(T)\|^{2}_{L_{1}}}\right)}&\leq K\left\|\left(\sum_{j=1}^{n}|\dc_{j}(T)|^{2}\right)^{1/2}\right\|_{L_{1}}\\
&\leq K\left(\sum_{j=1}^{n}\|\dc_{j}(T)\|^{2}_{L_{2}}\right)^{1/2}.
\end{split}
\end{equation}
We assume from now on that
\begin{equation}\label{assumption}
\sum_{j=1}^{n}\|\dc_{j}(T)\|^{2}_{L_{2}}\leq \frac{\varepsilon\log(n)}{16 K^{2}}
\end{equation}
for sufficient large $n$; otherwise, we obtain 
\[
\max_{j\in [n]}\|\dc_{j}(T)\|^{2}_{L_{2}}\geq \frac{\varepsilon\log(n)}{16 K^{2}n}.
\]
Substituting the assumption \eqref{assumption} to \eqref{E-G estimate} yields
\begin{equation*}\label{quantum KKL 3}
\frac{1}{4}\sqrt{\log\left(1+\frac{1}{\sum_{j=1}^{n}\|\dc_{j}(T)\|^{2}_{L_{1}}}\right)}\leq \frac{\sqrt{\varepsilon\log(n)}}{4},
\end{equation*}
which further implies
\[
\frac{1}{\sum_{j=1}^{n}\|\dc_{j}(T)\|^{2}_{L_{1}}}\leq n^{\varepsilon},
\]
and consequently,
\[
\frac{1}{n^{(1+\varepsilon)/2}}\leq \max_{j\in [n]}\|\dc_{j}(T)\|_{L_{1}}.
\]
It remains to choose $C=\min\{\frac{1}{16 K^{2}},1\}$ to obtain the desired result.
\end{proof}
In order to apply Theorem \ref{quantum KKL II} to obtain another KKL-type inequality, we introduce the concept of index for elements in $\M_{2^{n}}$ as follows. For each $T\in\M_{2^{n}}$, we define
\[
\mathrm{ind}(T)\coloneqq \inf\{\alpha\geq 0:\|\dc_{j}(T)\|^{\alpha}_{L_{1}}\leq \|\dc_{j}(T)\|^{2}_{L_{2}},~\forall~j\in[n]\}.
\]

\begin{lemma}\label{lem:alpha}
For each balanced projection $T\in\M_{2^{n}}$, we have $\mathrm{ind}(T)\in[1,2]$.
\end{lemma}
\begin{proof}
For balanced projection $T\in \mathbb{M}_{2^n}$, we note that
\begin{equation}\label{balanced condition}
\|\dc_{j}(T)\|_{L_{1}}\leq \|T\|_{L_{1}}=\frac{1}{2}<1,\quad\forall~j\in [n].
\end{equation}
If $\mathrm{ind}(T)>2$, there exists $\mathrm{ind}(T)\geq\alpha>2$ and some $j_{0}\in [n]$ such that
\[
\|\dc_{j_{0}}(T)\|_{L_{1}}^{\alpha}>\|\dc_{j_{0}}(T)\|^{2}_{L_{2}}\geq \|\dc_{j_{0}}(T)\|^{2}_{L_{1}}.
\]
This implies that $\|\dc_{j_{0}}(T)\|^{\alpha-2}_{L_{1}}>1$, which contradicts to \eqref{balanced condition}. Hence, $\mathrm{ind}(T)\leq 2$.

It remains to show that $\mathrm{ind}(T)\geq 1$. Indeed, since $\dc_{j}$ is a contraction, it follows from \eqref{balanced condition} that
\[
\left\|\dc_{j}(T)\right\|^{2}_{L_{2}}\leq\left\|\dc_{j}(T)\right\|_{L_{1}}\left\|\dc_{j}(T)\right\|_{L_{\infty}}\leq \left\|\dc_{j}(T)\right\|_{L_{1}}<1.
\]
On the other hand, since  $\|\dc_{j}(T)\|^{\alpha}_{L_{1}}$ is decreasing on $\alpha$ by \eqref{balanced condition}, we infer that $\mathrm{ind}(T)\geq 1$.
\end{proof}

We now conclude this subsection with the following quantum KKL-type inequality invoking $L_{2}$-influence.

\begin{theorem}\label{index version}
For each $n\in \mathbb{N}$ and balanced projection $T\in \M_{2^{n}}$ with $\mathrm{ind}(T)<2$, there exists a constant $C_{\mathrm{ind}(T)}>0$ depending on the index of $T$ such that
\begin{equation}\label{index KKL}
\max_{j\in [n]}\|\dc_{j}(T)\|^{2}_{L_{2}}\geq \frac{C_{\mathrm{ind}(T)}\log(n)}{n}.
\end{equation}
\end{theorem}
\begin{proof}
Choose positive $\alpha$ such that $\mathrm{ind}(T)\leq \alpha<2$ and set $\delta=\frac{2-\alpha}{4}$, $\varepsilon=\frac{2-\alpha}{2\alpha}$. By Theorem \ref{quantum KKL II}, there exists a universal constant $C>0$ such that one of the following inequalities holds
\begin{enumerate}[\rm(i)]
\item $\max_{j\in[n]}\|\dc_{j}(T)\|^{2}_{L_{2}}\geq  \frac{C\varepsilon\log(n)}{n}$;
\item $\max_{j\in[n]}\|\dc_{j}(T)\|_{L_{1}}\geq \frac{C}{n^{(1+\varepsilon)/2}}$.
\end{enumerate}
If the first situation holds, there is nothing to prove. We assume from now on that item (ii) holds, that is,
\begin{equation}\label{dichotomy}
\max_{j\in[n]}\|\dc_{j}(T)\|_{L_{1}}\geq \frac{C}{n^{(1-\delta)/\alpha}},
\end{equation}
where we used the fact $\varepsilon=\frac{2-2\delta-\alpha}{\alpha}$.
Since $\|\dc_{j}(T)\|^{\alpha}_{L_{1}}\leq \|\dc_{j}(T)\|^{2}_{L_{2}}$, it follows from \eqref{dichotomy} that
\[
\max_{j\in [n]}\|\dc_{j}(T)\|^{2}_{L_{2}}\geq\frac{C^{\alpha}}{n^{1-\delta}}\geq \frac{(2-\alpha)C^{\alpha}\log(n)}{4n}.
\]
Choosing $C_{\mathrm{ind}(T)}=\min\{\frac{C(2-\alpha)}{2\alpha},\frac{(2-\alpha)C^{\alpha}}{4}\}$ yields the desired inequality.
\end{proof}

\begin{remark}
There exists a quantum Boolean function $T\in \mathbb{M}_{2^n}$ with $\mathrm{ind}(T)=2$ such that the quantum KKL inequality holds for $L_{2}$-influence; see \cite[Proposition 11.5]{MO2010}.  However, as shown in \cite[Remark 6.5]{JLZ2025}, this is not the case in the CAR algebra setting.
\end{remark}

Our final application of Theorem \ref{quantum E-G} is the following stability result for the quantum KKL-type inequality (invoking $L_{1}$-influences), which is quantum analogy of \cite[Corollary 3.5]{EKLM2022}. For each $T\in \M_{2^{n}}$, we define that $|\nabla(T)|\coloneqq \left(\sum_{j=1}^{n}|\dc_{j}(T)|^{2}\right)^{1/2}$.

\begin{corollary}
Suppose that there exists a constant $C_{1}>0$ such that for each projection $T\in \M_{2^{n}}$ the following holds
	\begin{equation}\label{assumption of aftco1}
		\max_{j\in[n]} \|\dc_j(T)\|_{L_1}\leq \frac{C_{1}\log(n)\var(T)}{n}.
	\end{equation}
	Then there exist constant $C_2>0$  such that 
	\[
	\tr\left[\mathds{1}_{(\frac{1}{2}\var(T)\sqrt{\log(n)},\infty)}(|\nabla(T)|)\right]\geq C_{2}\var(T).
	\]
\end{corollary}

\begin{proof}
	By assumption \eqref{assumption of aftco1} and Proposition \ref{dlili} (i), we have
	\begin{align}\label{aftco eq4}
		\tr\left[\left|\nabla(T)\right|^{2}\right]=\sum_{j=1}^{n}\left\|\dc_{j}(T)\right\|^{2}_{L_{2}}\leq\sum_{j=1}^n\|\dc_jT\|_{L_1}\leq C_1\log(n)\var(T).
	\end{align}
	According to the assumption, there exists a universal constant $K_{1}>0$ such that
	\begin{equation}\label{aftco1 eq}
		\sum_{j=1}^n \|\dc_j(T)\|_{L_1}^2\leq \frac{C^{2}_{1}\log^2(n)\var(T)^{2}}{n}\leq \frac{K_{1}}{\sqrt{n}},
	\end{equation}
	where we used $\var(T)\leq \frac{1}{4}$ for the projection $T\in \M_{2^{n}}$.
	Using Theorem \ref{quantum E-G} and \eqref{aftco1 eq}, there exists $K_{2}>0$ such that
	\begin{equation}\label{aftco1 eq2}
		\left\|\nabla(T)\right\|_{L_{1}}\geq \frac{1}{K} \var(T)\sqrt{\log\left(1+\frac{1}{\sum_{j=1}^n \|\dc_j(T)\|_{L_1}^2}\right)}\geq K_{2} \var(T)\sqrt{\log(n)}.
	\end{equation}
	By the Paley-Zygmund inequality \eqref{eq:Paley-Zygmud}, there exists a constant $C_{2}>0$ such that
	\begin{equation*}
		\begin{aligned}
			&\tr\left[\mathds{1}_{(\frac12\var(T)\sqrt{\log(n)},\infty)}\left(|\nabla(T)|\right)\right]\geq\frac{\|\nabla(T)\|^{2}_{L_{1}}}{4\tr[|\nabla(T)|^{2}]}\geq C_{2}\var(T),
		\end{aligned}
	\end{equation*}
	where the last inequality follows from a combination of \eqref{aftco eq4} and \eqref{aftco1 eq2}.
\end{proof}

\subsection{Proof of Proposition \ref{key lemma 1}}

The proof of  Proposition \ref{key lemma 1} relies on the following technical lemmas. Firstly, with the moment comparison lemma (i.e., Lemma \ref{lemmaHC}) at hand, we can derive the following deviation inequality. Although this deviation inequality is well-known, we include its proof for the convenience of readers.

\begin{lemma}\label{lammaEstimate1_T}
There exists universal constant $K>0$ such that for each $T\in \M_{2^{n}}$ with degree at most $d\in\mathbb{N}$ and $\|T\|_{L_{2}}\leq 1$, we have
\[
\tr\left[\mathds{1}_{[t,\infty)}(|T|)\right]\leq K\exp\left\{-\frac{d\cdot t^{2/d}}{4e}\right\},\quad\mbox{for all }t>0.
\]
\end{lemma}
\begin{proof}
Applying the assumption $\|T\|_{L_{2}}=1$ and Lemma \ref{lemmaHC} it follows that
\begin{equation}\label{Lp growth}
\|T\|^{r}_{L_{r}}\leq r^{dr/2}\quad\mbox{for each }r\geq 1.
\end{equation}
Let $\alpha=\frac{d}{4e}$, and we now show that $\tr\left[\exp(\alpha|T|^{2/d})\right]<\infty$. Indeed, by the Taylor expansion and the Stirling formula $k!\sim \frac{k^{k}\sqrt{s\pi k}}{e^{k}}$, we get that
\begin{align*}
\tr\left[\exp(\alpha|T|^{2/d})\right]&=\sum_{k=0}^{\infty}\frac{\alpha^{k}\|T\|^{2k/d}_{L_{2k/d}}}{k!}\\
&\leq\sum_{k=0}^{\infty}\frac{\left(\frac{2\alpha}{d}\right)^{k}\cdot k^{k}}{k!}\\
&\leq K_{1}\sum_{k=0}^{\infty}\left(\frac{2e\alpha}{d}\right)^{k}=2K_{1}<\infty,
\end{align*}
where we used \eqref{Lp growth} and the Stirling formula in the second and the third inequality, respectively. Therefore, by the Chebyshev inequality, we have that for each $t>0$ the following holds
\[
\tr[\mathds{1}_{(t,\infty)}(|T|)]\leq e^{-\alpha t^{2/d}}\cdot\tr\left[\exp(\alpha|T|^{2/d})\right]\leq K\exp\left\{-\frac{d\cdot t^{2/d}}{4e}\right\}.
\]
\end{proof}

Applying the functional calculus for positive element in $\M_{2^{n}}$, we obtain the following integral representation lemma.
\begin{lemma}\label{lemmaIntT}
Let $S$, $T\in \M_{2^{n}}$. If $T$ is positive, then
\[
\tr(ST)=\int_{0}^\infty\tr\left(S\cdot\mathds{1}_{(t,\infty)}(T)\right).
\]
\end{lemma}

\begin{proof}
	By the functional calculus of $T$, it is clear that $T=\int_{0}^{\infty}\mathds{1}_{T}((t,\infty))dt$. We apply the integral representation of $T$ to the $\tr(ST)$ entails that
	\[
	\tr(ST)=\tr\left(S\cdot\int_{0}^{\infty}\mathds{1}_{(t,\infty)}(T)dt\right)=\int^{\infty}_{0}\tr\left(S\cdot\mathds{1}_{(t,\infty)}(T)\right)dt,
	\]
	where the last equality follows from the linearity of $\tr$.
\end{proof}

We also need the following estimate from \cite[Lemma 12]{KK2013}.
\begin{lemma}\label{KK18_5calculation}
	Let $d\geq1$ be a positive integer and $t_0>\left(4e\right)^{\frac{d}{2}}$. Then we have
	\[
	\int_{t_0}^\infty t^2\cdot \exp\left\{-\frac{d\cdot t^{2/d}}{2e}\right\}dt\leq 5et_0^{3-\frac{2}{d}}\exp\left\{-\frac{d\cdot t_0^{2/d}}{2e}\right\}.
	\]
\end{lemma}

We shall prove Proposition \ref{key lemma 1} in details. 
We here explain some notation we use below. For each $\s=(s_1,,\cdots,s_n)\in \{0,1,2,3\}^n$, it is viewed an element in $\{0,1,2,3\}^{n+1}$ via 
$$\widetilde{\s}=(s_1,\cdots,s_n,0).$$ 
To simplify symbols, we still write $\s$ instead of $\widetilde{\s}$.
Hence, for each $\s\in \{0,1,2,3\}^n$, the summation $\s\oplus e_{n+1}^{\alpha}$ is read as follows 
$$\s\oplus e_{n+1}^{\alpha}:=\widetilde{\s} \oplus e_{n+1}^{\alpha}\in \{0,1,2,3\}^{n+1}, \quad \alpha\in \{0,1,2,3\}.$$
For each $\s\in \{0,1,2,3\}^n$, we set
$$\s^{j\curvearrowright}=(s_1,\cdots,s_{j-1},0,s_{j+1},\cdots,s_n,s_j)\in \{0,1,2,3\}^{n+1},$$
i.e., we remove the original $ s_j $ to the $ n+1$-position and replace the original $s_j$ with 0.
By the same spirit, for each $\s\in \{0,1,2,3\}^n$, $\sigma_{\widetilde{\s}}=\sigma_{\s}\otimes \mathbf{1}_2$ is viewed as an element in $\M_{2^{n+1}}$.
For simplicity, we still write $\sigma_{\s}$ instead of $\sigma_{\widetilde{\s}}$. 

In what follows, let $d\geq1$ be fixed, $T\in\M_{2^{n}}$ and $J\subseteq [n]$. For $j\in J$, define 
\begin{equation}\label{eq:Tj}
	T_{j}\coloneqq \sum_{\substack{\supp(\s)\subseteq J^{c}\\|\supp(\s)|=d-1}}  \sum_{\alpha\in\{1,2,3\}}\widehat{T}(\s\oplus e^{\alpha}_{j})\sigma_{\s\oplus e^{\alpha}_{n+1}},
\end{equation}
\begin{equation}\label{T copyj}
	T_{\mathrm{copy},j}\coloneqq \sum_{\substack{\s\in\{0,1,2,3\}^{n}\\s_{j=0}}}\widehat{T}(\s)\sigma_{\s\oplus e^{0}_{n+1}}+\sum_{\substack{\s\in\{0,1,2,3\}^{n}\\s_{j\neq 0}}}\widehat{T}(\s)\sigma_{\s^{j\curvearrowright}},
\end{equation}
and 
\begin{equation}\label{Tilde Tj}
	\widetilde{T}_j\coloneqq \underbrace{\sum_{\substack{\s\in\{0,1,2,3\}^{n}\\s_{j=0}}}\widehat{T}(\s)\sigma_{\s\oplus e^{0}_{n+1}}}_{\widetilde{T}_{j,L}}+\underbrace{\sum_{\substack{\s\in\{0,1,2,3\}^{n}\\s_{j\neq 0}}}\widehat{T}(\s)\sigma_{\s\oplus e^{s_{j}}_{n+1}}}_{\widetilde{T}_{j,R}}.
\end{equation}
Then, $T_j$, $	T_{\mathrm{copy},j}$ and $\widetilde{T}_j$ are elements in $\M_{2^{n+1}}$.

Define $\Psi_j:\M_{2^{n}}\otimes\mathbf{1}_{2}\to\M_{2}^{\otimes j-1}\otimes\mathbf{1}_{2}\otimes\M_{2}^{\otimes n-j}$ by setting: for each $\s=(s_j)_{j=1}^n\in \{0,1,2,3\}^n$, 
\[
\sigma_{\s}\otimes\mathbf{1}_{2}\mapsto \sigma_{\s^{j\curvearrowright}}.
\]
It is clear that $\Psi_j$ is an $*$-isomorphism from $\M_{2^{n}}\otimes\mathbf{1}_{2}$ onto $\M_{2}^{\otimes j-1}\otimes\mathbf{1}_{2}\otimes\M_{2}^{\otimes n-j}$ with
\[
\Psi_j(T\otimes \mathbf{1}_{2})=T_{\mathrm{copy},j}
\]
and
\[
\Psi_j\left(\dc_{j}(T\otimes \mathbf{1}_{2})\right)=\dc_{n+1}(T_{\mathrm{copy},j})
\]
Hence, we have
\begin{equation}\label{key identity8}
	\left\|\dc_{n+1}(T_{\mathrm{copy},j})\right\|_{L_{1}}=\left\|\Psi_j\left(\dc_{j}(T\otimes \mathbf{1}_{2})\right)\right\|_{L_{1}}=\left\|\dc_{j}(T\otimes \mathbf{1}_{2})\right\|_{L_{1}}=\left\|\dc_{j}(T)\right\|_{L_{1}}.
\end{equation}

The next several technical lemmas provide necessary information of $T_j$, $T_{\mathrm{copy},j}$ and $\widetilde{T}_j$, which are key ingredients of proving Proposition \ref{key lemma 1}.

\begin{lemma}\label{lem cJc}
	Let $T\in \M_{2^{n}}$ and $J\subseteq [n]$. For each $j\in J$, we have
	$$	\mathcal{E}_{J^{c}\cup\{n+1\}}\left(A_j\widetilde{T}_j\right)=\mathcal{E}_{J^{c}\cup\{n+1\}}\left(\dc_{n+1}(T_{\mathrm{copy},j})\right),$$
	where $\widetilde{T}_j$ and $T_{\mathrm{copy},j}$ are as in \eqref{Tilde Tj} and \eqref{T copyj}, and $A_j$ is given by 
	\begin{equation}\label{Aj}
		A_j=\sum_{\alpha\in\{1,2,3\}}\sigma_{e^{\alpha}_{j}}.
	\end{equation}
\end{lemma}

\begin{proof} 
For fixed $j\in J$, according to the definition of conditional expectation, we note that for each $\widetilde{\s}\in \{0,1,2,3\}^{n+1}$, if $\widetilde{\s}_{i}\neq 0$ for some $i\in J$, then 
	\begin{equation}\label{note 1}
		\mathcal{E}_{J^c\cup \{n+1\}}(\sigma_{\widetilde{\s}})=0.
	\end{equation}
	From this, we immediately deduce that $	\mathcal{E}_{J^{c}\cup\{n+1\}}\left(A_j\widetilde{T}_{j,L}\right)=0$, where $\widetilde{T}_{j,L}$ is as in \eqref{Tilde Tj}. Thus,	
	\begin{equation}
		\label{key identity7}
		\begin{aligned}
				\mathcal{E}_{J^{c}\cup\{n+1\}}\left(A_j\widetilde{T}_j\right)&=\mathcal{E}_{J^{c}\cup\{n+1\}}\left(A_j\widetilde{T}_{j,R}\right)\\
			&=\sum_{\alpha\in \{1,2,3\}}\sum_{\substack{\s\{0,1,2,3\}^n\\s_{j}\neq 0}}\widehat{T}(\s)\mathcal{E}_{J^{c}\cup\{n+1\}}\left(\sigma_{e^{\alpha}_{j}}\cdot \sigma_{\s\oplus e^{s_j}_{n+1}}\right)\\
			&=\sum_{\alpha\in \{1,2,3\}}\sum_{\substack{\supp(\s)\subseteq J^{c}\cup \{j\}\\s_{j}\neq 0}}\widehat{T}(\s)\mathcal{E}_{J^{c}\cup\{n+1\}}\left(\sigma_{e^{\alpha}_{j}}\cdot \sigma_{\s\oplus e^{s_j}_{n+1}}\right)\\
			&=\sum_{\substack{\supp(\s)\subseteq J^{c}\cup\{j\}\\s_{j}\neq 0}}\widehat{T}(\s)\sigma_{\s^{j\curvearrowright}},
		\end{aligned}
	\end{equation}
	where we used \eqref{note 1} twice in last two equality.

	On the other hand side, it follows from \eqref{partial derivative} that
	\begin{equation*}\label{key identity6}
		\dc_{n+1}(T_{\mathrm{copy},j})=\sum_{\substack{\s\in\{0,1,2,3\}^{n}\\s_{j\neq 0}}}\widehat{T}(\s)\sigma_{\s^{j\curvearrowright}}.
	\end{equation*}
	Hence, by \eqref{note 1} again,
	\begin{equation*}
		\mathcal{E}_{J^{c}\cup\{n+1\}}\left(\dc_{n+1}(T_{\mathrm{copy},j})\right)=	\sum_{\substack{\supp(\s)\subseteq J^{c}\cup\{j\}\\s_{j}\neq 0}}\widehat{T}(\s)\sigma_{\s^{j\curvearrowright}}.
	\end{equation*}
	The desired assertion follows from the above argument. 
\end{proof}

\begin{lemma}\label{lem Tj 1}
Let $T\in \M_{2^{n}}$ and $J\subset [n]$. For each $j\in J$, we have
\[
\sum_{\substack{\supp(\s)\subseteq J^{c}\\|\supp(\s)|=d-1,\alpha\in \{1,2,3\}}}|\widehat{T}(\s\oplus e^{\alpha}_{j})|^{2}=\left(\tr\left[\overline{T}_{j}\cdot A_j\widetilde{T}_j\right]\right)^{2},
\]
where $\overline{T}_j=T_j/\|T_j\|_{L_2}$, $T_j$, $\widetilde{T}_j$ and $A_j$ are referred to \eqref{eq:Tj}, \eqref{Tilde Tj} and \eqref{Aj}, respectively.
\end{lemma}

\begin{proof}
	Take $j\in J$. From the orthogonality of $\{\sigma_{\s}\}_{\s\in \{0,1,2,3\}^{n+1}}$ in $L_2(\M_{2^{n+1}})$, it is not hard to see that for each $\s\in\{0,1,2,3\}^n$ and $\alpha\in \{1,2,3\}$, 
	\begin{equation*}
		\left\langle \sigma_{\s\oplus e^{\alpha}_{n+1}},  A_j \widetilde{T}_{j,L}\right\rangle =0,
	\end{equation*}	
	where $\widetilde{T}_{j,L}$ is as in \eqref{Tilde Tj}.
	It follows that 
	\begin{equation*}
		\begin{aligned}
			&\left\langle T_j,  A_j \widetilde{T}_{j}\right\rangle 
=	\left\langle T_j,  A_j (\widetilde{T}_{j,L}+\widetilde{T}_{j,R})\right\rangle =	\left\langle T_j, A_j \widetilde{T}_{j,R}\right\rangle\\
=&\sum_{\substack{\supp(\s^{\prime})\subseteq J^{c}\\|\supp(\s^{\prime})|=d-1}} \sum_{\substack{\s\in\{0,1,2,3\}^{n}\\s_{j}\neq 0}} \sum_{\alpha^{\prime},\alpha\in\{1,2,3\}} \langle\widehat{T}(\s^{\prime}\oplus e^{\alpha^{\prime}}_{j})\sigma_{\s^{\prime}\oplus e^{\alpha^{\prime}}_{n+1}}, \widehat{T}(\s)\sigma_{e^{\alpha}_{j}}\cdot \sigma_{\s\oplus e^{s_{j}}_{n+1}}\rangle\\
=&\sum_{\substack{\supp(\s^{\prime})\subseteq J^{c}\\|\supp(\s^{\prime})|=d-1}}\sum_{\alpha'\in\{1,2,3\}}|\widehat{T}(\s^{\prime}\oplus e^{\alpha^{\prime}}_{j})|^{2},
		\end{aligned}
	\end{equation*}
	where in the second equality we used the orthogonality of $\{\sigma_{\s}\}_{\s\in \{0,1,2,3\}^{n+1}}$ in $L_2(\M_{2^{n+1}})$.
	To verify the desired assertion, it suffices to note that 
	\begin{equation*}\label{key identity1}
		\|T_{j}\|^{2}_{L_{2}}=\sum_{\substack{\supp(\s^{\prime})\subseteq J^{c}\\|\supp(\s^{\prime})|=d-1}}\sum_{\alpha'\in\{1,2,3\}}|\widehat{T}(\s^{\prime}\oplus e^{\alpha^{\prime}}_{j})|^{2}.
	\end{equation*}
\end{proof}

\begin{lemma}\label{Y1j}
	Let $T\in \M_{2^{n}}$ and $J\subseteq [n]$. For each $j\in J$ and $t_0>0$, we have
	\begin{align*}
		\int_{0}^{t_0}&\tr\left[\mathds{1}_{(t,\infty)}(\left|\overline{T}_{j}\right|)\cdot\left|\mathcal{E}_{J^{c}\cup\{n+1\}}\left(A_j\widetilde{T}_j\right)\right|\right]~dt\leq t_{0}\left\|\dc_{j}(T)\right\|_{L_{1}}
	\end{align*}
	where $\overline{T}_j=T_j/\|T_j\|_2$, $T_j$, $\widetilde{T}_j$ and $A_j$ are referred to \eqref{eq:Tj}, \eqref{Tilde Tj} and \eqref{Aj}, respectively.
\end{lemma}

\begin{proof}
It follows from Lemma \ref{lem cJc} that
\begin{equation*}\label{main estimate 2}
\begin{split}
&\int_{0}^{t_{0}}\tr\left[\mathds{1}_{(t,\infty)}(\left|\overline{T}_{j}\right|)\cdot\left|\mathcal{E}_{J^{c}\cup\{n+1\}}\left(A_j\widetilde{T}_j\right)\right|\right]~dt\\
\leq&t_{0}\left\|\mathcal{E}_{J^{c}\cup\{n+1\}}(\dc_{n+1}(T_{\mathrm{copy},j}))\right\|_{L_{1}}\leq t_{0}\left\|\dc_{n+1}(T_{\mathrm{copy},j})\right\|_{L_{1}},
\end{split}
\end{equation*}
where the second inequality is due to the fact that conditional expectation is bounded on $L_p$, $1\leq p\leq \infty$. The desired inequality follows from  \eqref{key identity8}.
\end{proof}

\begin{lemma}\label{Y2j}
Let $T\in \M_{2^{n}}$ and $J\subseteq [n]$. For each $j\in J$ and $t_0>(2e)^{\frac{d}{2}}$, we have
\begin{align*}
&\int_{t_0}^{\infty}\tr\left[\mathds{1}_{(t,\infty)}(\left|\overline{T}_{j}\right|)\cdot\left|\mathcal{E}_{J^{c}\cup\{n+1\}}\left(A_j\widetilde{T}_j\right)\right|\right]~dt\\
\leq&\sqrt{5e}t_0^{1-\frac{1}{d}}\exp\left\{-\frac{d\cdot t_0^{2/d}}{4e}\right\}\left(\sum_{\substack{\supp(\s)\subseteq J^{c}\cup\{j\}\\s_{j}\neq 0}}|\widehat{T}(\s)|^{2}\right)^{1/2},
\end{align*}
where $\overline{T}_j=T_j/\|T_j\|_2$, $T_j$, $\widetilde{T}_j$ and $A_j$ are referred to \eqref{eq:Tj}, \eqref{Tilde Tj} and \eqref{Aj}, respectively.
\end{lemma}

\begin{proof}
Applying the Cauchy-Schwarz inequality twice, we get
	\begin{equation*}\label{main estimate 4}
		\begin{split}
			&\int_{t_{0}}^{\infty}\tr\left[\mathds{1}_{(t,\infty)}(\left|\overline{T}_{j}\right|)\cdot\left|\mathcal{E}_{J^{c}\cup\{n+1\}}\left(A_j\widetilde{T}_j\right)\right|\right]dt\\
			\leq&\left(\int_{t_{0}}^{\infty}\frac{1}{t^{2}}~dt\right)^{1/2}\cdot \left(\int_{t_{0}}^{\infty}t^{2}\cdot \left(\tr\left[\mathds{1}_{(t,\infty)}(\left|\overline{T}_{j}\right|)\cdot\left|\mathcal{E}_{J^{c}\cup\{n+1\}}\left(A_j\widetilde{T}_j\right)\right|\right]\right)^{2}dt\right)^{1/2}\\
			\leq&\frac{1}{\sqrt{t_{0}}}\left(\int_{t_{0}}^{\infty}t^{2}\tr\left[\mathds{1}_{(t,\infty)}(\overline{T}_{j})\right]dt\right)^{1/2}\cdot\left\|\mathcal{E}_{J^{c}\cup\{n+1\}}\left(A_j\widetilde{T}_j\right)\right\|_{L_{2}}.
		\end{split}
	\end{equation*}
Combining Lemma \ref{lammaEstimate1_T} and Lemma \ref{KK18_5calculation}, we have
	\begin{equation*}
	\int_{t_{0}}^{\infty}t^{2}\tr\left[\mathds{1}_{(t,\infty)}(\overline{T}_{j})\right]dt\leq 5et_0^{3-\frac{2}{d}}\exp\left\{-\frac{d\cdot t_0^{2/d}}{2e}\right\}.
	\end{equation*}
 Furthermore, according to \eqref{key identity7}, we have
\[
\left\|\mathcal{E}_{J^{c}\cup\{n+1\}}\left(A_j\widetilde{T}_j\right)\right\|_{L_{2}}^2=\sum_{\substack{\supp(\s)\subseteq J^{c}\cup\{j\}\\s_{j}\neq 0}}|\widehat{T}(\s)|^{2},\quad\mbox{for each }j\in J.
\]
The desired assertion follows.
\end{proof}
We now are ready to provide the proof of Proposition \ref{key lemma 1}.
\begin{proof}[Proof of Proposition \ref{key lemma 1}]
For a given projection $T\in\M_{2^{n}}$, it is clear that $\|T\|_2\leq 1$. We assume without loss of generality that $T=\sum_{\s\in\{0,1,2,3\}^{n}}\widehat{T}(\s)\sigma_{\s}$.
For each $j\in J$, by Lemma \ref{lem Tj 1}, we have
\begin{align*}
	\sum_{\substack{\supp(\s)\subseteq J^{c}\\|\supp(\s)|=d-1,\alpha\in \{1,2,3\}}}|\widehat{T}(\s\oplus e^{\alpha}_{j})|^{2}&=\left(\tr\left[\overline{T}_{j}\cdot A_j\widetilde{T}_j\right]\right)^{2}\\
	&=\left(\tr\left[\mathcal{E}_{J^c\cup\{n+1\}}(\overline{T}_{j}\cdot A_j\widetilde{T}_j)\right]\right)^{2}\\
	&=\left(\tr\left[\overline{T}_{j}\mathcal{E}_{J^c\cup\{n+1\}}( A_j\widetilde{T}_j)\right]\right)^{2},
\end{align*}
where we used that each conditional expectation preserves trace and $\mathcal{E}_{J^c\cup\{n+1\}}(T_{j})=T_j$ (this follows from the definition of $T_j$ as in \eqref{eq:Tj}). Using Lemma \ref{lemmaIntT} we have
\begin{align*}
\left(\tr\left[\overline{T}_{j}\mathcal{E}_{J^c\cup\{n+1\}}( A_j\widetilde{T}_j)\right]\right)^{2}
			&\leq 	\left(\tr\left[|\overline{T}_{j}|\cdot|\mathcal{E}_{J^c\cup\{n+1\}}( A_j\widetilde{T}_j)|\right]\right)^{2}\\
			&=\left\{\int_{0}^{\infty}\tr\left[\mathds{1}_{(t,\infty)}(\left|\overline{T}_{j}\right|)\cdot|\mathcal{E}_{J^c\cup\{n+1\}}( A_j\widetilde{T}_j)|\right]~dt\right\}^{2}\\
			&\leq 2\left\{\int_{0}^{t_{0}}\tr\left[\mathds{1}_{(t,\infty)}(\left|\overline{T}_{j}\right|)\cdot|\mathcal{E}_{J^c\cup\{n+1\}}( A_j\widetilde{T}_j)|\right]~dt\right\}^{2}\\
			&\quad+2\left\{\int_{t_{0}}^{\infty}\tr\left[\mathds{1}_{(t,\infty)}(\left|\overline{T}_{j}\right|)\cdot|\mathcal{E}_{J^c\cup\{n+1\}}( A_j\widetilde{T}_j)|\right]~dt\right\}^{2}\\
			&\coloneqq 2\mathsf{Y}_{1,j}(T)^2+2\mathsf{Y}_{2,j}(T)^2,
\end{align*}
where $t_0>0$ is chosen to satisfy
\begin{equation}\label{specific choice}
\exp\left\{-\frac{d\cdot t_{0}^{2/d}}{2e}\right\}=\sum_{j\in J}\|\dc_{j}(T)\|^{2}_{L_{1}}=M_J(T).
\end{equation}
Note here that \eqref{specific choice} and the assumption $M(T)\leq e^{-2d}$ imply that
\[
t^{2}_{0}=\left(\frac{2e}{d}\right)^{d}\left(\log\left(\frac{1}{M_J(T)}\right)\right)^{d}\quad\mbox{and}\quad t_{0}\geq (4e)^{\frac{d}{2}}.
\]
We conclude from the above argument that 
\begin{equation}\label{lhs key lem}
	\sum_{j\in J}	\sum_{\substack{\supp(\s)\subseteq J^{c}\\|\supp(\s)|=d-1,\alpha\in \{1,2,3\}}}|\widehat{T}(\s\oplus e^{\alpha}_{j})|^{2}\leq 2\sum_{j\in J}\mathsf{Y}_{1,j}(T)^2+2\sum_{j\in J}\mathsf{Y}_{2,j}(T)^2.
\end{equation}
	
By Lemma \ref{Y1j}, we have 
\begin{align*}
\sum_{j\in J}\mathsf{Y}_{1,j}(T)^2\leq t_{0}^2\sum_{j\in J}\left\|\dc_{j}(T)\right\|_{L_{1}}^2=\left(\frac{2e}{d}\right)^{d}M_J(T)\left(\log\left(\frac{1}{M_J(T)}\right)\right)^{d}.
\end{align*}
Using Lemma \ref{Y2j}, we have 
\begin{align*}
\sum_{j\in J}\mathsf{Y}_{2,j}(T)^2&\leq 5et_0^{2-\frac{2}{d}}\exp\left\{-\frac{d\cdot t_0^{2/d}}{2e}\right\} \sum_{j\in J}\sum_{\substack{\supp(\s)\subseteq J^{c}\cup\{j\}\\s_{j}\neq 0}}|\widehat{T}(\s)|^{2}\\
&=5et_0^{-\frac{2}{d}}\left(\frac{2e}{d}\right)^{d}M_J(T)\left(\log\left(\frac{1}{M_J(T)}\right)\right)^{d} \sum_{j\in J}\sum_{\substack{\supp(\s)\subseteq J^{c}\cup\{j\}\\s_{j}\neq 0}}|\widehat{T}(\s)|^{2}\\
&\leq 2\left(\frac{2e}{d}\right)^{d}M_J(T)\left(\log\left(\frac{1}{M_J(T)}\right)\right)^{d}.
\end{align*}
Here we also used the fact (note that $T$ is a projection, and hence $\|T\|_{L_2}\leq 1$)
\[
\sum_{j\in J}\sum_{\substack{\supp(\s)\subseteq J^{c}\cup\{j\}\\s_{j}\neq 0}}|\widehat{T}(\s)|^{2}\leq \|T\|_{L_2}^2\leq 1.
\]

Substituting the estimates of $\sum_{j\in J} \mathsf{Y}_{1,j}^2$ and $\sum_{j\in J} \mathsf{Y}_{2,j}^2$ to \eqref{lhs key lem}, we get 
\begin{equation*}
\sum_{j\in J}	\sum_{\substack{\supp(\s)\subseteq J^{c}\\|\supp(\s)|=d-1,\alpha\in \{1,2,3\}}}|\widehat{T}(\s\oplus e^{\alpha}_{j})|^{2}\leq 6\left(\frac{2e}{d}\right)^{d}M_J(T)\left(\log\left(\frac{1}{M_J(T)}\right)\right)^{d}.
\end{equation*}
This completes the proof of the proposition.	
\end{proof}

\section*{Acknowledgements}
Yong Jiao is partially supported by the National Key R$\&$D Program of China (No. 2023YFA1010800) and the NSFC (Nos. 12125109, W2411005). Sijie Luo is partially supported by the NSFC (No. 12201646) and the Natural Science Foundation Hunan (No. 2023JJ40696). Dejian Zhou is partially supported by the NSFC (No. 12471134), the Natural Science Foundation Hunan (No. 2023JJ20058), and the CSU Innovation-Driven Research Programme (No. 2023CXQD016).

\end{document}